\documentclass{amsart}

\usepackage{amsfonts}
\usepackage{amsmath,amsthm,amssymb}
\usepackage{hyperref}
\usepackage{graphicx,color}
\numberwithin{equation}{section}
\usepackage{tikz,ytableau}
\usetikzlibrary{decorations.pathreplacing}

\newtheorem{thm}{Theorem}[section]
\newtheorem{lem}[thm]{Lemma}
\newtheorem{prop}[thm]{Proposition}
\newtheorem{cor}[thm]{Corollary}
\theoremstyle{definition}

\newtheorem{defn}[thm]{Definition}
\newtheorem{conj}[thm]{Conjecture}

\newtheorem{problem}[thm]{Problem}
\newtheorem{remark}[thm]{Remark}

\DeclareMathOperator*{\KK}{\text{\Large \bf K}}

\newcommand\PP{\mathcal{P}}

\newcommand\LL{\mathcal{L}}

\newcommand\Mot{\operatorname{Motz}}
\newcommand\wt{\operatorname{wt}}

\newcommand\flr[1]{\left\lfloor #1\right\rfloor}

\newcommand\Qbinom[3]{\genfrac{[}{]}{0pt}{}{#1}{#2}_{#3}}
\newcommand\qbinom[2]{\Qbinom{#1}{#2}{q}}

\newcommand\area{\operatorname{area}}
\newcommand\col{\operatorname{col}}
\newcommand\row{\operatorname{row}}

\newcommand\qhyper[5]{
  {}_{#1}\phi_{#2} \left(#3;#4;#5\right)
}

\title{Three families of $q$-Lommel polynomials}
\author{Jang Soo Kim and Dennis Stanton}
\thanks{The first author was supported by NRF grants \#2019R1F1A1059081 and \#2016R1A5A1008055.} 
\address{Department of Mathematics,
Sungkyunkwan University (SKKU), Suwon, Gyeonggi-do 16419, South Korea}
\email{jangsookim@skku.edu}
\address{School of Mathematics,
University of Minnesota,
Minneapolis, Minnesota 55455, USA}
\email{stanton@math.umn.edu}
\date{\today}

\begin{document}

\begin{abstract} Three $q$-versions of Lommel polynomials are studied. Included
  are explicit representations, recurrences, continued fractions, and
  connections to associated Askey--Wilson polynomials. Combinatorial
  results are emphasized, including a general theorem when $R_I$ moments coincide with
  orthogonal polynomial moments. The combinatorial results use weighted Motzkin
  paths, Schr\"oder paths, and parallelogram polyominoes.
\end{abstract}

\maketitle

\section{Introduction}

Lehmer \cite{Lehmer_1944} used the following Bessel function identity to study zeros of Bessel functions
\begin{equation}
\label{eq:Lehmereq}
\frac{J_{\nu+1}(x)}{J_\nu(x)}=2 \sum_{n=1}^\infty \sigma_{2n}(\nu) x^{2n-1}.
\end{equation}
In this identity $\sigma_{2n}(\nu)$ is the $2n^{\mathrm{th}}$ power sum of the inverses of the 
positive zeros $j_{\nu,k}$ of $J_\nu(x)$,
\begin{equation}
  \label{eq:sigma2n}
  \sigma_{2n}(\nu)=\sum_{k=1}^\infty j_{\nu,k}^{-2n}.
\end{equation}
Lehmer noted that $\sigma_{2n}(\nu)$ is a rational function of $\nu$, with a predictable 
denominator, and a numerator with nonnegative coefficients.
Kishore \cite{Kishore1964} proved Lehmer's positivity conjecture. Lalanne 
\cite[Prop. 3.6]{Lalanne92}, \cite[Th. 4.7]{Lalanne93}
proved $q$-versions of Kishore's result using weighted binary trees and also weighted Dyck paths.

The Lommel polynomials are orthogonal with respect to the linear functional
$$
\LL(P(x))=\sum_{k=1}^\infty  \left( P\left(j_{\nu,k}^{-1}\right) + P\left(-j_{\nu,k}^{-1}\right) \right)j_{\nu,k}^{-2}.
$$
Thus \( \sigma_{2n}(\nu) \) in \eqref{eq:sigma2n} is effectively the \(
(2n-2)^{\mathrm{th}}\) moment for the Lommel polynomials, while \eqref{eq:Lehmereq}
is the Lommel moment generating function.

The purpose of this paper is to study two sets of $q$-Lommel orthogonal
polynomials, whose moment generating functions are quotients of $q$-Bessel
functions. We also consider another set of polynomials, which is a type $R_I$
polynomial, and whose moment generating function is again a quotient of
$q$-Bessel functions.
Koelink and Van Assche \cite{Koelink1995} and  Koelink \cite{Koelink1999}
analytically studied two of these \( q \)-Lommel polynomials.
In this paper we concentrate on the combinatorial aspect of these three 
\( q \)-Lommel polynomials.

There are combinatorial results on the quotient of Bessel functions
and the quotient of $q$-Bessel functions. Delest and F\'edou \cite{Delest1993} showed that a
generating function for parallelogram polyominoes can be written as a ratio of
Jackson's third $q$-Bessel functions. Bousquet-M\'elou and Viennot
\cite{Bousquet-Melou_1992} generalized their result by adding one more
parameter. A recounting of the history of the combinatorics of the $q$-analogue
of the quotient of Bessel functions may be found in \cite[Sec.
1]{Bousquet-Melou1995} (see also \cite[Sec. 4]{Lalanne93}). It includes results
by Klarner and Rivest \cite[see (19)]{Klarner74}, Delest and Fedou
\cite{Delest1993}, Fedou \cite{Fedou_1994} Lalanne \cite{Lalanne92,Lalanne93},
Brak and Guttman \cite{Brak1990}, Bousquet-M\'elou and Viennot
\cite{Bousquet-Melou_1992}, and Barcucci et al. \cite[Cor. 3.5]{Barcucci1998},
\cite[Th. 4.3, Th. 5.3]{Barcucci_2001}.

In this paper we put these results in perspective by relating them to \( q
\)-Lommel polynomials.
The moment generating function has a continued fraction expansion.
Using the general theory of orthogonal and type $R_I$
polynomials we give finite versions of the infinite continued fractions. We show
that a generating function for bounded diagonal parallelogram polyominoes is
given by a ratio of \( q \)-Lommel polynomials, which is a finite version of the
result of Bousquet-M\'elou and Viennot \cite{Bousquet-Melou_1992}.

Even though the Lommel polynomials have a hypergeometric representation as a ${}_2F_3$,
they do not appear in the Askey scheme. In this paper we rectify this, by realizing two sets of 
$q$-Lommel polynomials as limiting cases of associated Askey--Wilson
polynomials.
One may ask for an associated Askey scheme which contains this limiting case.

The paper is organized in the following way. In
Section~\ref{sec:q-lommel-polynomials} we define the three sets of $q$-Lommel
polynomials using three-term recurrence relations. The classical connection
between these polynomials and $q$-Bessel functions is given in
Section~\ref{sec:q-bessel-functions}. The associated Askey--Wilson polynomials
are reviewed in Section~\ref{sec:q-lommel-polynomials-1}, along with explicit
limiting cases to the $q$-Lommel polynomials, see Theorems~\ref{thm:limit1} and
\ref{thm:limit2}. In Section~\ref{sec:moments-cont-fract} we independently prove
the continued fraction expansions for the moment generating functions, and give
two surprising equalities of continued fractions in Corollary~\ref{cor:bigcor}
and Theorem~\ref{thm:genequalmom}. Combinatorial interpretations of these
continued fractions are given in Section~\ref{sec:ratios-q-lommel}, see
Theorem~\ref{thm:ratio of R-1} and Corollary~\ref{cor:double sum ratio of R-1}.
A general combinatorial result for the concurrence of type $R_I$ moments and orthogonal
polynomials moments is given in Section~\ref{sec:concurrence-moments},
see Theorem~\ref{thm:concurmom}. In
Section~\ref{sec:open-problems} we propose some open problems.

We use the standard notations for both hypergeometric series and basic
hypergeometric series \cite{GR}.

\section{$q$-Lommel polynomials}
\label{sec:q-lommel-polynomials}

In this section we give the defining recurrence relations for the Lommel, 
the classical $q$-Lommel, the even-odd $q$-Lommel, and the type 
$R_I$ $q$-Lommel
polynomials.

\begin{defn} 
\label{defn:Lommel}
The \emph{monic Lommel polynomials} $h_n(x;c)$ are defined by 
$$
h_{n+1}(x;c)=xh_n(x;c)- \frac{1}{(c+n)(c+n-1)}h_{n-1}(x;c), n\ge 0,
\quad h_{-1}(x;c)=0, h_0(x;c)=1.
$$
\end{defn}

We consider three versions of $q$-Lommel polynomials.

\begin{defn}\cite[\S 14.4]{Ismail} The \emph{classical $q$-Lommel polynomials} are defined by
$$
h_{n+1}(x;c,q)=xh_n(x;c,q)-\lambda_nh_{n-1}(x;c,q), \quad n\ge 0, \quad h_{-1}(x;c,q)=0, \quad h_0(x;c,q)=1,
$$
where
\[
  \lambda_{n} = \frac{cq^{n-1}}{(1-cq^{n-1})(1-cq^{n})}.
\]
\end{defn}

\begin{defn} 
\label{defn:evenodd}
The \emph{even-odd $q$-Lommel polynomials} are defined by
$$
p_{n+1}(x;c,q)=xp_n(x;c,q)-\lambda_np_{n-1}(x;c,q), \quad n\ge 0, \quad p_{-1}(x;c,q)=0, \quad p_0(x;c,q)=1,
$$
where
\begin{equation}
  \lambda_{2n} = \frac{cq^{3n-1}}{(1-cq^{2n-1})(1-cq^{2n})},
  \qquad
  \lambda_{2n+1} = \frac{q^{n}}{(1-cq^{2n})(1-cq^{2n+1})}.
\end{equation}
\end{defn}

Note that 
$$
\lim_{q\to 1} (1-q)^n h_n(x/(1-q);q^c,q)=h_n(x;c), \quad \lim_{q\to 1} (1-q)^n p_n(x/
(1-q);q^c,q)=h_n(x;c),
$$
so that each polynomial may be considered as a $q$-analogue of the classical Lommel polynomials.

\begin{defn} 
\label{defn:typeI}
The \emph{type $R_I$ $q$-Lommel polynomials} are defined by
$$
r_{n+1}(x;c,q)=(x-b_n)r_n(x;c,q)-xa_nr_{n-1}(x;c,q), \quad r_{-1}(x;c,q)=0, \quad r_0(x;c,q)=1,
$$
where
\[
  b_n = \frac{q^{n}}{1-cq^{n}}, \qquad
  a_n = \frac{cq^{2n-1}}{(1-cq^{n-1})(1-cq^{n})}.
\]
\end{defn}

Note that if 
$$
\hat{r}_n(x;c)=\lim_{q\to 1} (1-q)^{2n} r_n(x/(1-q)^2;q^c,q),
$$
then
\begin{equation}
\label{eq:RR}
\hat{r}_{n+1}(x;c)=x \hat{r}_n(x;c)-\frac{x}{(c+n-1)(c+n)}\hat{r}_{n-1}(x;c).
\end{equation}

The polynomials $\hat{r}_n(x;c)$ in \eqref{eq:RR} are closely related to the
monic Lommel polynomials. For example it is known that
their moments are the same, see \eqref{eq:1}. 

Koelink and Van Assche study the even-odd and the type $R_I$ $q$-Lommel polynomials
in  \cite[Sec. 4]{Koelink1995}, and Koelink continues this analytic study 
in \cite{Koelink1999}.

Orthogonality relations for the classical $q$-Lommel are in \cite[Theorem 14.4.3]{Ismail}, while those 
for the even-odd $q$-Lommel and the type $R_I$ $q$-Lommel are in
\cite[Theorem 4.2]{Koelink1995} and \cite[Theorem 3.4]{Koelink1995}.

\section{$q$-Bessel functions and $q$-Lommel polynomials}
\label{sec:q-bessel-functions}

In this section we give the recurrence relation which connects $q$-Bessel
functions to the classical $q$-Lommel polynomials and the type $R_I$ $q$-Lommel polynomials.
This was the original motivation for Lommel polynomials.

\begin{defn}
The \emph{Bessel functions} $J_\nu(x)$ are defined by
\[
  J_\nu(z) =\frac{(z/2)^\nu}{\Gamma(\nu+1)} \sum_{n\ge0} \frac{(-z^2/4)^n}{n! (\nu+1)_n}.
\]
\end{defn}

\begin{defn}\cite[p.188, (6.2)]{Chihara}
The \emph{classical Lommel polynomials} $R_{n,\nu}(z)$ are (non-monic) polynomials in $z^{-1}$ defined by
\( R_{0,\nu}(z) =1 \), \( R_{1,\nu}(z) = 2\nu/z \), and
\begin{equation}
\label{eq:rec R}
R_{n+1,\nu}(z) = \frac{2(n+\nu)}{z} R_{n,\nu}(z) - R_{n-1,\nu}(z), \qquad n\ge1. 
\end{equation}
Equivalently, 
$$
h_n(x;c)=R_{n,c}(2/x)/(c)_n.
$$
\end{defn}

The connection of Bessel functions to Lommel polynomials is the following proposition.

\begin{prop}\cite[p.187]{Chihara}
\label{prop:LomBes}
The Bessel functions and the classical Lommel polynomials are related by the recurrence
\begin{equation}\label{eq:Bessel-Lommel}
  J_{\nu+n}(z) = R_{n,\nu}(z) J_\nu(z) -R_{n-1,\nu+1}(z) J_{\nu-1}(z).
\end{equation}
\end{prop}

\begin{defn}
\emph{Jackson's first $q$-Bessel function} $J^{(1)}_\nu(z;q)$ and
\emph{second $q$-Bessel function} $J^{(2)}_\nu(z;q)$ are defined by
\begin{align}
\label{eq:J1}  J^{(1)}_\nu(z;q)
  &= \frac{(q^{\nu+1};q)_\infty}{(q;q)_\infty} (z/2)^\nu \qhyper21{0,0}{q^{\nu+1}}{q,-z^2/4},\\
\label{eq:J2}  J^{(2)}_\nu(z;q)
  &= \frac{(q^{\nu+1};q)_\infty}{(q;q)_\infty} (z/2)^\nu \qhyper01{-}{q^{\nu+1}}{q,-q^{\nu+1}z^2/4}.
\end{align}
\end{defn}

In this paper we consider only the first and third $q$-Bessel function, as the second 
$q$-Bessel can be obtained from the first by changing $q$ to $q^{-1}.$ Recall that we 
consider formal power series in $z$, and have no restriction on $q$.

\begin{prop}\cite[(14.4.1)]{Ismail} 
\label{prop:qBesLom1}
The first $q$-Bessel functions satisfy  
\begin{equation}
\label{eq:qBesLom1}
q^{n\nu+\binom{n}{2}} J_{\nu+n}^{(1)}(x;q)= R_{n,\nu}^{(1)}(x;q)J_{\nu}^{(1)}(x;q)
-R_{n-1,\nu+1}^{(1)}(x;q)J_{\nu-1}^{(1)}(x;q).
\end{equation}
where
\( R_{0,\nu}^{(1)}(x;q)=1 \),  \( R_{1,\nu}^{(1)}(x;q)=2(1-q^{n+\nu})/x \), and
$$
\frac{2}{x} (1-q^{n+\nu})
R_{n,\nu}^{(1)}(x;q)=R_{n+1,\nu}^{(1)}(x;q)+q^{n+\nu-1}R_{n-1,\nu}^{(1)}(x;q),
\qquad n\ge 1.
$$
\end{prop}

Again we need a rescaling to obtain the classical $q$-Lommel polynomials,
$$
h_n(x;c,q)=R_{n,\nu}^{(1)}(2/x;q)/(q^\nu;q)_n, \quad c=q^\nu.
$$

\begin{defn}
\label{defn:third}
The \emph{Jackson's third $q$-Bessel functions} $J_\nu^{(3)}(z;q)$ are defined by
\[
  J_\nu^{(3)}(z;q) = \frac{(q^{\nu+1};q)_\infty z^\nu}{(q;q)_\infty}
  \qhyper11{0}{q^{\nu+1}}{q,qz^2}.
\]  
\end{defn}

Define the Laurent polynomials \( R_{m,\nu}^{(3)}(z;q) \) by
\begin{equation}
  \label{eq:R}
  R_{m+1,\nu}^{(3)}(z;q) = \left( z + z^{-1}(1-q^{\nu+m}) \right) R_{m,\nu}^{(3)}(z;q)
  -R_{m-1,\nu}^{(3)}(z;q).
\end{equation}
We rescale these Laurent polynomials to obtain polynomials
\begin{equation}
  \label{eq:r=R}
  r^{(3)}_n(x;c,q):=\frac{x^{n/2}}{(q^{-\nu};q^{-1})_n} R_{n,\nu}^{(3)}(x^{-1/2};q^{-1}), \quad c=q^\nu.
\end{equation}
Then \( r^{(3)}_n(x;c,q) \) are the type $R_I$  polynomials
defined by
\( r^{(3)}_{-1}(x;c,q)=0\), \( r^{(3)}_0(x;c,q)=1 \), and 
\begin{equation}
  \label{eq:r3}
  r^{(3)}_{n+1}(x;c,q)=(x-\hat b_n)r^{(3)}_n(x;c,q)-x \hat a_nr^{(3)}_{n-1}(x;c,q), \qquad n\ge0,
\end{equation}
where
\[
 \hat b_n = \frac{cq^{n}}{1-cq^{n}}, \qquad
 \hat a_n = \frac{c^2q^{2n-1}}{(1-cq^{n-1})(1-cq^{n})}.
\]
Using the recurrences one can easily check that
\[
  r_{n}(x;c,q)  = \frac{r^{(3)}_{n}(cx;c,q)}{c^n} ,
\]
where \( r_{n}(x;c,q) \) are the type \( R_I \) \( q \)-Lommel polynomials
$r_n(x;c,q)$ in Definition~\ref{defn:typeI}.

Koelink and Swarttouw \cite[(4.12)]{Koelink_1994} showed that the third $q$-Bessel functions satisfy the
following property analogous to \eqref{prop:LomBes} and
\eqref{eq:qBesLom1}.

\begin{prop} 
The third $q$-Bessel functions satisfy
\begin{equation}
\label{eq:KS rec2}
  J_{\nu+m}^{(3)}(z;q) = R_{m,\nu}^{(3)}(z;q) J_{\nu}^{(3)}(z;q)
  -R_{m-1,\nu+1}^{(3)}(z;q) J_{\nu-1}^{(3)}(z;q).
\end{equation}
\end{prop}

Koelink and Swarttouw \cite[(4,24)]{Koelink_1994} also showed that
\begin{equation}
  \label{eq:KSlim}
  \lim_{m\to\infty} z^m R^{(3)}_{m,\nu}(z;q)
  = \frac{(q;q)_\infty z^{1-\nu}}{(z^2;q)_\infty} J^{(3)}_{\nu-1}(z;q),
\end{equation}
which implies
\begin{equation}\label{eq:R=J}
  \lim_{m\to\infty}  \frac{R^{(3)}_{m,\nu+2}(z;q)}{R^{(3)}_{m+1,\nu+1}(z;q)}
  = \frac{J^{(3)}_{\nu+1}(z;q)}{J^{(3)}_\nu(z;q)} .
\end{equation}
By \eqref{eq:r=R} and \eqref{eq:R=J} we have
\begin{equation}\label{eq:r=J}
   \frac{J^{(3)}_{\nu+1}(x^{1/2};q^{-1})}{J^{(3)}_\nu(x^{1/2};q^{-1})}= 
   \lim_{n\to\infty}  \frac{-q^{\nu+1}r^{(3)}_{n}(x^{-1};q^{\nu+2},q)}{x^{1/2}(1-q^{\nu+1})r^{(3)}_{n+1}(x^{-1};q^{\nu+1},q)}.
\end{equation}

The $q$-Bessel function relation for the 
even-odd $q$-Lommel polynomials which corresponds to Proposition~\ref{prop:qBesLom1} 
is given in \cite[Proposition 4.1]{Koelink1995}.
 
\section{$q$-Lommel polynomials and the Askey scheme}
\label{sec:q-lommel-polynomials-1}

The $q$-Lommel polynomials do not appear in the Askey scheme. 
In this section we realize both the classical $q$-Lommel and the even-odd $q$-Lommel 
polynomials as limiting cases of the associated Askey--Wilson polynomials, see 
Theorems~\ref{thm:limit1} and \ref{thm:limit2}. 
We then use results of Ismail and Masson \cite{IsmailMasson} to give explicit formulas for each polynomial.
Finally we prove that the moments for even-odd $q$-Lommel and the 
type $R_I$ $q$-Lommel agree, see Theorem~\ref{thm:equalmom}.

An explicit formula for the Lommel polynomial \( h_n(x;c) \) is 
\[
h_n(x;c)=x^n{}_2F_3(-n/2,(1-n)/2;c,1-c-n,-n;-4/x^2).
\]
In this section we give explicit formulas for our three families of \( q
\)-Lommel polynomials. The classical \( q \)-Lommel polynomials have a
corresponding single sum formula \cite[Theorem~14.4.1]{Ismail}:
\[
  h_{n}(x ;c, q)= \frac{1}{(c;q)_n}
  \sum_{j=0}^{\lfloor n / 2\rfloor} \frac{(-1)^{j}\left(c, q ; q\right)_{n-j}}
  {\left(q, c ; q\right)_{j}(q ; q)_{n-2 j}}x^{n-2 j} c^jq^{j(j-1)}.
\]

Here are the main results for the even-odd $q$-Lommel polynomials.

\begin{thm} 
\label{thm:explicit1}
The even even-odd $q$-Lommel polynomials have the explicit formula
\begin{align*}
p_{2n}(x;c,q) &=(-1)^n\frac{q^{\binom{n}{2}}}{(c;q)_{2n}}
\sum_{k=0}^n \frac{(q^{-n},cq^n,c;q)_k}{(q;q)_k}q^{k} x^{2k}\\
&\quad\times \sum_{s=0}^{n-k} \frac{(cq^{k-1};q)_s}{(q;q)_s} 
\frac{1-cq^{k-1+2s}}{1-cq^{k-1}} \frac{(cq^{n+k},q^{k-n},q^{k};q)_s}{(q^{-n},cq^{n},c;q)_s}
c^{s} q^{-sk+s(s-1)}.\\
\end{align*}
\end{thm}

\begin{thm} 
\label{thm:explicit2}
The odd even-odd $q$-Lommel polynomials have the explicit formula
\begin{align*}
p_{2n+1}(x;c,q) &=(-c)^n\frac{q^{n^2+\binom{n+1}{2}}}{(cq;q)_{2n}}
\sum_{k=0}^n \frac{(q^{-n},cq^{n+1},cq;q)_k}{(q;q)_k}c^{-k}q^{-k^2} x^{2k+1}\\
&\quad\times \sum_{s=0}^{n-k} \frac{(cq^{k};q)_s}{(q;q)_s} 
\frac{1-cq^{k+2s}}{1-cq^{k}} \frac{(cq^{n+k+1},q^{k-n},q^{k+1};q)_s}{(q^{-n},cq^{n+1},c;q)_s}
c^{s} q^{-(3k+2)s-s(s-1)}.\\
\end{align*}
\end{thm}

\begin{proof}
First we write the even even-odd polynomials as orthogonal polynomials in $x^2$ using the 
odd-even trick.  Then we realize the new polynomials as limiting cases of 
associated Askey--Wilson polynomials, for which explicit formulas are known.
The same method will work for the odd even-odd polynomials.

We begin with the associated Askey--Wilson polynomials.
The monic Askey--Wilson polynomials satisfy
\begin{equation}
p_{n+1}(x) = (x-b_n) p_n(x) - \lambda_n p_{n-1}(x), \quad n\ge 1,
\end{equation}
where 
\[
  b_n = \frac{1}{2} (a+a^{-1}-A_n-C_n), \qquad
  \lambda_n = \frac{1}{4} A_{n-1} C_n,
\]
\begin{align*}
  A_n &= \frac{(1-abq^n)(1-acq^n)(1-adq^n)(1-abcdq^{n-1})}{a(1-abcdq^{2n-1})(1-abcdq^{2n})},\\
  C_n &= \frac{a(1-q^n)(1-bcq^{n-1})(1-bdq^{n-1})(1-cdq^{n-1})}{(1-abcdq^{2n-2})(1-abcdq^{2n-1})}.
\end{align*}

The associated Askey--Wilson polynomials replace $q^n$ by $\alpha q^n$ in the three-term recurrence relation.

\begin{defn}
  The \emph{associated Askey--Wilson polynomials} $p_n^{(\alpha)}(x)$ are defined as a solution to 
\begin{equation}
\label{eq:alphaAW}
p_{n+1}^{(\alpha)}(x) = (x-b_n(\alpha)) p_n^{(\alpha)}(x) - \lambda_n(\alpha) p_{n-1}^{(\alpha)}(x), \quad n\ge 1,
\end{equation}
\[
  b_n(\alpha) = \frac{1}{2} (a+a^{-1}-A_n(\alpha)-C_n(\alpha)), \qquad
  \lambda_n(\alpha) = \frac{1}{4} A_{n-1}(\alpha) C_n(\alpha),
\]
\begin{align*}
  A_n(\alpha,q) &= \frac{(1-ab\alpha q^n)(1-ac\alpha q^n)(1-ad\alpha q^n)(1-abcd\alpha q^{n-1})}
        {a(1-abcd \alpha^2q^{2n-1})(1-abcd \alpha^2q^{2n})},\\
  C_n(\alpha,q) &= \frac{a(1-\alpha q^n)(1-bc\alpha q^{n-1})(1-bd\alpha q^{n-1})(1-cd\alpha q^{n-1})}
        {(1-abcd \alpha^2q^{2n-2})(1-abcd \alpha^2q^{2n-1})}.
\end{align*}
\end{defn}

There are two linearly independent solutions to \eqref{eq:alphaAW}, depending on the initial conditions.
Ismail and Rahman \cite[(4.15), (8.9)]{Ismail} gave these two independent solutions 
as double sums, the inner sum a very well poised $ _{10}W_9$. 

\begin{thm} 
\label{thm:IM}
Two linearly independent solutions $\psi_n^{(\alpha,\epsilon)}(x,q), \epsilon=1,2$ 
to \eqref{eq:alphaAW} are given by
\begin{align*}
&\psi_n^{(\alpha,\epsilon)}(x;q)=  K_n\sum_{k=0}^n \frac{(q^{-n},abcd\alpha^2 q^{n-1}, abcd\alpha^2/q,az,a/z;q)_k}
{(q,ab\alpha,ac\alpha,ad\alpha, abcd\alpha/q;q)_k} q^k\\
&
\times{}_{10}W_9
\left( abcd\alpha^2 q^{k-2};\alpha, bc\alpha/q, bd\alpha/q, cd\alpha/q, S,abcd\alpha^2q^{n+k-1},q^{k-n};q; T\right)
\end{align*}
where 
$$
K_n= (2a)^{-n} \frac{(ab\alpha, ac\alpha, ad\alpha, abcd\alpha/q;q)_n}
{(abcd\alpha^2q^{n-1},abcd\alpha^2/q;q)_n}
$$
and the two choices for $\epsilon$ correspond to 
$$
(S,T)= (q^{k+1},a^2), {\text{ for $\epsilon=1,$ }} (S,T)= (q^{k},qa^2),{ \text{ for $\epsilon=2.$}}
$$
\end{thm}

We next explain how Theorem~\ref{thm:explicit1} follows from Theorem~\ref{thm:IM}.

First we rewrite the recurrence relation \cite{Chihara} in terms of polynomials in $x^2$. 

\begin{prop} 
\label{prop:firstprop}
If $p_{2n}(x;c,q)=t_n(x^2),$ then
$$
t_{n+1}(x)=(x-B_n)t_n(x)-\Lambda_n t_{n-1}(x), \quad t_{-1}=0, \quad t_0(x)=1.
$$
where
\begin{align*}
B_0&=\frac{1}{(1-c)(1-cq)},\\
 B_n &= \lambda_{2n}+\lambda_{2n+1},
 \quad n\ge 1,\\
\Lambda_n &= \lambda_{2n-1}\lambda_{2n},
\quad n\ge 1.
\end{align*}
\end{prop}

 \begin{prop} 
 \label{prop:secondprop}
 If $p_{2n+1}(x;c,q)=xs_n(x^2),$ then
$$
s_{n+1}(x)=(x-B_n)s_n(x)-\Lambda_n s_{n-1}(x), \quad s_{-1}=0, \quad s_0(x)=1.
$$
where
\begin{align*}
 B_n &= \lambda_{2n+2}+\lambda_{2n+1},
 \quad n\ge 1,\\
\Lambda_n &= \lambda_{2n+1}\lambda_{2n},
\quad n\ge 1.
\end{align*}
\end{prop}

We shall obtain the recurrence relations in Propositions~\ref{prop:firstprop} and ~\ref{prop:secondprop}
by an appropriate limiting case of Theorem~\ref{thm:IM}. Our goal is to obtain 
$(A_n,C_n)=(\lambda_{2n+1},\lambda_{2n})$ for $t_n(x)$ and 
$(A_n,C_n)=(\lambda_{2n+2},\lambda_{2n+1})$ for $s_n(x).$ Then we match the 
initial conditions to find the correct linear combination of the two solutions.

First choosing $a=c^{-1}q^{-1}\alpha$, $b=c=d=1/\alpha$, we obtain
\begin{align*}
  A_n(\alpha,1/q) &= \frac{\alpha(1-c q^{n+1}/\alpha)^3(1-\alpha cq^n)}{cq(1-c q^{2n})(1-c q^{2n+1})},\\
  C_n(\alpha,1/q) &= \frac{cq(1-q^n/\alpha )(1-\alpha q^{n-1})^3}{\alpha(1-c q^{2n-1})(1-c q^{2n})}.
\end{align*}

By rescaling $x$ by $B\alpha^2 x/2$, i.e., $\hat p_n(x) = 2^n \alpha^{-2n} B^{-n}p_n^{(\alpha)}(B\alpha^2 x/2)$,
we have
\[
\hat p_{n+1}(x) = (x-\hat b_n(\alpha)) \hat p_n(x) - \hat \lambda_n(\alpha) \hat p_{n-1}(x),
\]
\begin{align*}
 \hat b_n(\alpha) &=   \frac{1}{B\alpha^2}\left( \frac{cq}{\alpha}+\frac{\alpha}{cq} 
  -A_n(\alpha,1/q)-C_n(\alpha,1/q)\right),\\
  \hat\lambda_n(\alpha) &= 
 \frac{1}{B^2\alpha^4}A_n(\alpha,1/q)C_n(\alpha,1/q).
\end{align*}
If $\alpha\to \infty$, the first two terms in $\hat b_n(\alpha)$ vanish. 
Choosing $B=1/q$, we obtain the desired values for Proposition~\ref{prop:firstprop}
\begin{align*}
\lim_{\alpha\to\infty}  \frac{-1}{B\alpha^2} A_n(\alpha,1/q) &=\frac{q^n}{(1-cq^{2n})(1-cq^{2n+1})}= \lambda_{2n+1},\\
\lim_{\alpha\to\infty}  \frac{-1}{B\alpha^2} C_n(\alpha,1/q) &= \frac{cq^{3n-1}}{(1-cq^{2n-1})(1-cq^{2n})}= \lambda_{2n}.
\end{align*}

The first degree limiting polynomial matches the second Ismail-Rahman solution  
in Theorem~\ref{thm:IM} with
$(a,b,c,d)=(\alpha/cq,1/\alpha,1/\alpha,1/\alpha)$, 
$$
x-\frac{1}{(1-c)(1-cq)}
$$
so that 
$$
\lim_{\alpha\to\infty} \hat p_n(x)= \lim_{\alpha\to\infty}\psi_n^{(\alpha,2)}( x;1/q),
$$
which is the stated explicit formula in Theorem~\ref{thm:explicit1}.
\end{proof}

For the odd even-odd polynomials in Proposition~\ref{prop:secondprop}, we choose
$(a,b,c,d)=(cq^{2}\alpha,1/\alpha,1/\alpha,1/\alpha)$,

\begin{align*}
  A_n(\alpha,q) &= \frac{(1-c \alpha q^{n+2})^3(1-cq^{n+1}/\alpha)}{\alpha cq^2(1-c q^{2n+1})(1-c q^{2n+2})},\\
  C_n(\alpha,q) &= \frac{\alpha cq^2(1-\alpha q^n)(1-q^{n-1}/\alpha)^3}{(1-c q^{2n})(1-c q^{2n+1})}.
\end{align*}
As before choosing $\hat p_n(x) = 2^n \alpha^{-2n} B^{-n}p_n^{(\alpha)}(B\alpha^2 x/2)$ and $B=-cq^2$ we find

\begin{align*}
\lim_{\alpha\to\infty}  \frac{1}{B\alpha^2} A_n(\alpha,q) &=\frac{cq^{3n+2}}{(1-cq^{2n+1})(1-cq^{2n+2})}= \lambda_{2n+2},\\
\lim_{\alpha\to\infty}  \frac{1}{B\alpha^2} C_n(\alpha,q) &= \frac{q^{n}}{(1-cq^{2n})(1-cq^{2n+1})}= \lambda_{2n+1}.
\end{align*}

The first degree limiting polynomial matches the first Ismail--Rahman solution  
in Theorem~\ref{thm:IM} with
$(a,b,c,d)=(cq^2\alpha,1/\alpha,1/\alpha,1/\alpha)$, 
$$
x-\frac{1+cq}{(1-c)(1-cq^2)}
$$
so that 
$$
\lim_{\alpha\to\infty} \hat p_n(x)= \lim_{\alpha\to\infty}\psi_n^{(\alpha,1)}(x;q),
$$
which is the stated explicit formula in Theorem~\ref{thm:explicit2}.

We summarize these limits for the even-odd $q$-Lommel polynomials.
 
\begin{thm}
\label{thm:limit1} The even-odd $q$-Lommel polynomials are the following limits of 
associated Askey--Wilson polynomials
\begin{align*}
p_{2n}(x;c,q) &=\lim_{\alpha\to\infty}\frac{(2q)^n}{\alpha^{2n}} \psi_n^{(\alpha,2)}(\alpha^2 x^2/2q;1/q), 
\quad (a,b,c,d)=(\alpha/cq,1/\alpha,1/\alpha,1/\alpha),\\
p_{2n+1}(x;c,q)&= x\lim_{\alpha\to\infty}\frac{(-2/cq^2)^n}{\alpha^{2n}} \psi_n^{(\alpha,1)}(-cq^2\alpha^2 x^2/2;q), 
\quad (a,b,c,d)=(cq^2\alpha,1/\alpha,1/\alpha,1/\alpha).
\end{align*}

\end{thm}

For the classical $q$-Lommel polynomials $h_n(x;c;q)$, for the even polynomials choose
\begin{align*}
(a,b,c,d)&=(1, q/\alpha^2,c,1/\alpha),\\
\hat p_n(x) &= 2^n \alpha^{-2n} (-1)^{n}p_n^{(\alpha)}(-\alpha^2 x/2),\\ 
t_{2n}(x)&= 2^n \alpha^{-2n} (-1)^{n}\lim_{\alpha\to\infty}\psi_n^{(\alpha,2)}( -\alpha^2 x^2/2;q) 
\end{align*}
and for the odd polynomials choose
\begin{align*}
(a,b,c,d)&=(1, q^2/\alpha^2,c,1/\alpha),\\
\hat p_n(x) &= 2^n \alpha^{-2n} (-q)^{n}p_n^{(\alpha)}(-\alpha^2 x/2q),\\ 
s_{2n}(x)&= 2^n \alpha^{-2n} (-q)^{n}\lim_{\alpha\to\infty}\psi_n^{(\alpha,1)}( -\alpha^2 x^2/2q;q). 
\end{align*}

\begin{thm}
\label{thm:limit2} The classical $q$-Lommel polynomials are the following limits of 
associated Askey--Wilson polynomials
\begin{align*}
h_{2n}(x;c,q)&=\lim_{\alpha\to\infty}\frac{(-2)^n}{\alpha^{2n}} \psi_n^{(\alpha,2)}(-\alpha^2 x^2/2;q), 
\quad (a,b,c,d)=(1, q/\alpha^2,c,1/\alpha),\\
h_{2n+1}(x;c,q)&= x\lim_{\alpha\to\infty}\frac{(-2q)^n}{\alpha^{2n}} \psi_n^{(\alpha,1)}(-\alpha^2 x^2/2q;q), 
\quad (a,b,c,d)=(1, q^2/\alpha^2,c,1/\alpha).
\end{align*}

\end{thm}

Theorem~\ref{thm:classexp} is \cite[Theorem~14.4.1]{Ismail}.

\begin{thm} 
\label{thm:classexp}
The classical $q$-Lommel polynomials are
$$
h_{n}(x;c,q)=
\sum_{k=0}^{n/2} \qbinom{n-k}{k}
\frac{(-c)^k q^{k^2-k}}
{(c;q)_k (cq^{n-1};q^{-1})_k}  x^{n-2k}.
$$
\end{thm}

\begin{proof} 
We consider the even case, the proof for the odd case is similar. 
The inner sum becomes an evaluable very well poised ${}_6W_5$
$$
{}_6W_5\left(cq^{k-1}; q^k, cq^{n+k}, q^{k-n}
\ \Bigr| \ q;q^{-2k}\right)
=\frac{(cq^{n-1};q^{-1})_{k}(q^{n+1};q)_{k}}{(c;q)_{k}(q^{k+1};q)_k} q^{-k(n-k)}.
$$ 
By considering the coefficient of $x^{2n-2k}$, we arrive at Theorem~\ref{thm:classexp}
with $n$ replaced by $2n$. The odd case actually gives the same result.
\end{proof}

For the type $R_I$ $q$-Lommel polynomials there is a simple generating function 
which gives an explicit expression.

\begin{prop} 
\label{prop:Igf}
The type $R_I$ $q$-Lommel polynomials have the generating function
$$
\sum_{n=0}^\infty (c^{-1};q^{-1})_n \ r_n(x;c,q) t^n=
\sum_{k=0}^\infty \frac{(-xt/c)^k q^{-\binom{k}{2}}}
{(t/c,tx;q^{-1})_{k+1}}.
$$
\end{prop}

\begin{proof} 
If $G(x,t)$ is the generating function on the left side, then
Definition~\ref{defn:typeI} implies
\begin{align*}
G(x,t)-1&= (x+1/c)t G(x,t)-xt/c \ G(x,tq^{-1})-xt^2/c \ G(x,t),\\
G(x,t)&=\frac{1}{(1-xt)(1-t/c)}-\frac{xt/c}{(1-xt)(1-t/c)}G(x,tq^{-1})
\end{align*}
whose iterate is the result.
\end{proof}

\begin{thm}
\label{thm:RIexplicit}
The type $R_I$ $q$-Lommel polynomials have the explicit formula
$$
r_n(x;c,q)=\frac{1}{(c^{-1};q^{-1})_n}\sum_{k=0}^n \sum_{a=0}^{n-k}
(-x/c)^k q^{-\binom{k}{2}}\Qbinom{k+a}{a}{q^{-1}} c^{-a}
\Qbinom{n-a}{k}{q^{-1}} x^{n-k-a}.
$$
\end{thm}

\begin{proof}
Apply the $q^{-1}$-binomial theorem to Proposition~\ref{prop:Igf} to find the 
resulting coefficient of $t^n$.
\end{proof}

\begin{prop} 
\label{prop:weirdconncoef}
We have the connection coefficient relation
$$
r_n(x^2;c,q)=\sum_{k=0}^n \Qbinom{n}{k}{q} \frac{c^k q^{n^2-(n-k)^2}}
{(cq^{n-1},cq^{2n-k};q^{-1})_k} p_{2n-2k}(x;c,q).
$$
\end{prop}

\begin{proof} Induction on $n$ using the three term relations.
\end{proof}

\begin{prop} 
\label{prop:momentsmatch}
If $L_p$ is the linear functional for the even-odd polynomials $p_n(x;c,q)$, then
$$
L_p(r_n(x^2;c,q))= \frac{c^nq^{n^2}}{(c,cq;q)_n}.
$$
\end{prop}

\begin{proof} Apply $L_p$ to both sides of Proposition~\ref{prop:weirdconncoef}.
By orthogonality, $L_p(p_j(x))=0$ for $j>0,$ so only the $k=n$ term survives. 
\end{proof}

\begin{thm}
\label{thm:equalmom} 
The moments of the type $R_I$ $q$-Lommel polynomials are equal to the even moments of 
the even-odd $q$-Lommel polynomials,
$$
L_{r}(x^m)=L_p(x^{2m}), \quad m\ge 0.
$$ 
\end{thm}

\begin{proof} The type $R_I$ moments $L_{r}(x^m)$ are recursively determined by
  \cite[Corollary 3.15]{kimstanton:R1}
$$
L_r(r_n(x;c,q))=a_1a_2\cdots a_n= \frac{c^nq^{n^2}}{(c,cq;q)_n}, \quad n\ge 0.
$$
By Proposition~\ref{prop:momentsmatch} the moments $L_p(x^{2m})$ satisfy the same recurrence.
\end{proof}

For completeness, we give the inverse relation to Proposition~\ref{prop:weirdconncoef}.

\begin{prop} 
\label{prop:invweirdconncoef}
We have the connection coefficient relation
$$
p_{2n}(x;c,q)=\sum_{k=0}^n \Qbinom{n}{k}{q} \frac{(-c)^k q^{2nk-\binom{k+1}{2}}}
{(cq^{n-1},cq^{2n-1};q^{-1})_k} r_{n-k}(x^2;c,q).
$$
\end{prop}

\begin{prop} The even-odd $q$-Lommel polynomials have the explicit expressions
\begin{align*}
p_{2n}(x;c,q)&=\frac{1}{(c;q)_{2n}}\sum_{k=0}^n 
(-1)^k x^{2n-2k}(cq^k;q)_{2n-2k}\sum_{j=0}^k \Qbinom{n-j}{k-j}{q}
\Qbinom{n-k+j-1}{j}{q} c^j 
q^{jn+\binom{k}{2}},\\
p_{2n+1}(x;c,q)&=\frac{1}{(c;q)_{2n+1}}\sum_{k=0}^n 
(-1)^k x^{2n-2k+1}(cq^k;q)_{2n-2k+1}\sum_{j=0}^k \Qbinom{n-j}{k-j}{q}
\Qbinom{n-k+j}{j}{q}  c^j 
q^{jn+\binom{k}{2}}.
\end{align*}

\end{prop}

\begin{proof} This may be verified from Definition~\ref{defn:evenodd}, by considering the 
coefficients of $x^{2n-2k-1}.$
\end{proof}

\section{Moments and Continued fractions}
\label{sec:moments-cont-fract}

In this section we review the known facts which connect
continued fractions to moment generating functions. We independently prove 
the continued fractions for the moments of each of the three $q$-Lommel polynomials.

\begin{defn}
  Take a sequence of orthogonal polynomials $p_n(x)$ which satisfy 
  \( p_{-1}(x)=0 \), \( p_0(x)=1 \), and
  $$
  p_{n+1}(x)=(x-b_n)p_n(x)-\lambda_n p_{n-1}(x), \quad n\ge 0,
  $$
  and whose linear functional for orthogonality is $L_p.$
  Define
  \[
    \mu_n(\{b_k\}_{k\ge0},\{\lambda_k\}_{k\ge0}) = L_p(x^n).
    \]
  The moment generating function for $L_p$ is
  $$
  \sum_{n=0}^\infty L_p(x^n) t^n= \sum_{n=0}^\infty \mu_n(\{b_k\}_{k\ge0},\{\lambda_k\}_{k\ge0}) t^n.
  $$
\end{defn}

A Jacobi continued fraction also exists for $M_p(t),$ converging as 
formal power series in $t$,

\begin{equation}
\label{eq:Jaccontfrac}
\sum_{n=0}^\infty L_p(x^n) t^n= 
\cfrac{1}{1-b_0t-\cfrac{\lambda_1t^2}{1-b_1t-\cfrac{\lambda_2t^2}{1-\cdots}}}.
\end{equation}

\begin{defn} \cite{kimstanton:R1}
  For general type $R_I$ orthogonal polynomials
  $$
  r_{n+1}(x)=(x-b_n)r_n(x)-(a_n x+\lambda_n) r_{n-1}(x),  \quad n\ge 0,
  $$
  with linear functional $L_r$, define
  \[
    \mu_n(\{b_k\}_{k\ge0},\{a_k\}_{k\ge0},\{\lambda_k\}_{k\ge0}) = L_r(x^n).
  \]
\end{defn}

The corresponding continued fraction for the type \( R_I \) moment generating
function is \cite[Corollary 3.7]{kimstanton:R1}
\begin{equation}
\label{eq:typeIcontfrac}
\sum_{n=0}^\infty L_r(x^n) t^n= 
\cfrac{1}{1-b_0t-\cfrac{a_1t+\lambda_1t^2}{1-b_1t-\cfrac{a_2t+\lambda_2t^2}{1-\cdots}}}.
\end{equation}

Note that both continued fractions in \eqref{eq:Jaccontfrac} and \eqref{eq:typeIcontfrac}
are explicitly given in terms of the three term recurrence coefficients. We shall evaluate
the continued fractions as quotients of basic hypergeometric series, namely $q$-Bessel functions, 
using contiguous relations.

For the Lommel polynomials $h_n(x;c)$, it is known that the moment generating function is a 
quotient of Bessel functions, with $\lambda_n=1/(c+n-1)(c+n),$
\[
  \sum_{n=0}^\infty L_h(x^n) t^n=\frac{{}_0F_1(c+1;-t^2)}{{}_0F_1(c;-t^2)} =
  \cfrac{1}{1-\cfrac{\lambda_1t^2}{1-\cfrac{\lambda_2t^2}{1-\cdots}}}.
\]

The moment generating function for the classical $q$-Lommel polynomials is a quotient of 
$q$-Bessel functions. In this section we shall see that a corresponding result holds for
our other two $q$-Lommel polynomials, and in fact they are equal.

\begin{thm}
\label{thm:firstmom}
\cite[Theorem~14.4.3]{Ismail}
The moment generating function for the classical $q$-Lommel polynomials $h_n(x;c,q)$ is 
a quotient of Jackson's first $q$-Bessel functions
\[
  \sum_{n=0}^\infty L_h(x^n) t^n=\frac{ \qhyper21{0,0}{cq}{q,-t^2}}{ \qhyper21{0,0}{c}{q;-t^2}}
  = \cfrac{1}{1-\cfrac{\lambda_1t^2}{1-\cfrac{\lambda_2t^2}{1-\cdots}}},
\]
with $\lambda_n=cq^{n-1}/(1-cq^{n-1})(1-cq^n)$.
\end{thm}

\begin{thm}
\label{thm:secondmom}
The moment generating function for the even-odd $q$-Lommel polynomials $p_n(x;c,q)$ is 
a quotient of Jackson's third $q$-Bessel functions
\[
  \sum_{n=0}^\infty L_p(x^n) t^n=\frac{ \qhyper11{0}{cq}{q;qt^2}}{ \qhyper11{0}{c}{q;t^2}}
  = \cfrac{1}{1-\cfrac{\lambda_1t^2}{1-\cfrac{\lambda_2t^2}{1-\cdots}}},
\]
with 
$$
  \lambda_{2n} = \frac{cq^{3n-1}}{(1-cq^{2n-1})(1-cq^{2n})},
  \qquad
  \lambda_{2n+1} = \frac{q^{n}}{(1-cq^{2n})(1-cq^{2n+1})}.
$$
\end{thm}

\begin{thm}
\label{thm:thirdmom}
The moment generating function for the type $R_I$ $q$-Lommel polynomials $r_n(x;c,q)$ is 
a quotient of Jackson's third $q$-Bessel functions
\[
  \sum_{n=0}^\infty L_r(x^n) z^n=\frac{ \qhyper11{0}{cq}{q;qz}}{ \qhyper11{0}{c}{q;z}}
  = \cfrac{1}{1-b_0z - \cfrac{a_1z}{1-b_1z-\cfrac{a_2z}{1-b_2z-\cdots}}},
\]
with 
$$
  a_{n} = \frac{cq^{2n-1}}{(1-cq^{n-1})(1-cq^{n})},
  \qquad
 b_{n} = \frac{q^{n}}{1-cq^{n}}.
$$
\end{thm}

Theorem~\ref{thm:equalmom} implies that the two continued fractions in
Theorems~\ref{thm:secondmom} and \ref{thm:thirdmom} with \( z=t^2 \) are equal.

\begin{cor} 
\label{cor:bigcor}
We have the equality of continued fractions
\[
 \cfrac{1}{1-b_0z - \cfrac{a_1z}{1-b_1z-\cfrac{a_2z}{1-b_2z-\cdots}}}
=  \cfrac{1}{1-\cfrac{\lambda_1z}{1-\cfrac{\lambda_2z}{1-\cdots}}},
\]
where
$$
  a_{n} = \frac{cq^{2n-1}}{(1-cq^{n-1})(1-cq^{n})},
  \qquad
 b_{n} = \frac{q^{n}}{1-cq^{n}},
$$
$$
  \lambda_{2n} = \frac{cq^{3n-1}}{(1-cq^{2n-1})(1-cq^{2n})},
  \qquad
  \lambda_{2n+1} = \frac{q^{n}}{(1-cq^{2n})(1-cq^{2n+1})}.
$$
\end{cor}

Theorems~\ref{thm:firstmom}, \ref{thm:secondmom}, and \ref{thm:thirdmom}  may 
all be proven using contiguous relations for hypergeometric and basic hypergeometric series.

To prove Theorems~\ref{thm:firstmom} and \ref{thm:secondmom} we use Heine's
contiguous relation \cite[17.6.19]{NIST:DLMF} which is
\[
  \qhyper21{aq,b}{cq}{q,z} - \qhyper21{a,b}{c}{q,z}
  = \frac{(1-b)(a-c)z}{(1-c)(1-cq)} \qhyper21{aq,bq}{cq^2}{q,z}.
\]
Equivalently, 
\begin{equation}
  \label{eq:Heine_contigous2}
  \frac{\qhyper21{aq,b}{cq}{q,z}}{\qhyper21{a,b}{c}{q,z}}
  =\cfrac{1}{1- \cfrac{(1-b)(a-c)z}{(1-c)(1-cq)}
  \cdot \cfrac{\qhyper21{bq,aq}{cq^2}{q,z}}{\qhyper21{b,aq}{cq}{q,z}}}.
\end{equation}
Applying \eqref{eq:Heine_contigous2} iteratively, we obtain Heine's continued
fraction, which is a $q$-analogue of Gauss's continued fraction.

\begin{lem}[Heine's fraction]
  \label{lem:Heine}
We have  
\[
  \frac{\qhyper21{aq,b}{cq}{q,z}}{\qhyper21{a,b}{c}{q,z}} =
  \cfrac{1}{1-\cfrac{\beta_1z}{1-\cfrac{\beta_2z}{1-\cdots}}},
\]
where
\[
  \beta_{2n+1} = \frac{(1-bq^n)(a-cq^n)q^{n}}{(1-cq^{2n})(1-cq^{2n+1})},
  \qquad
  \beta_{2n} = \frac{(1-aq^n)(b-cq^n)q^{n-1}}{(1-cq^{2n-1})(1-cq^{2n})}.
\]
\end{lem}

Theorem~\ref{thm:firstmom} is the special case \( a=b=0 \) and \( z=-t^2 \) of
Lemma~\ref{lem:Heine}. Theorem~\ref{thm:secondmom} is also the limiting case
$z=-t^2/a,$ $b=0,$ $a\to\infty$ of Lemma~\ref{lem:Heine}.

For Theorem~\ref{thm:thirdmom} we need the $q$-N\"orlund fraction \cite[(19.2.7)]{Cuyt}.
However, to simplify the expressions we need some notation for continued fractions.

\begin{defn}
For sequences $a_i$ and $b_i$, let
\[
 \KK_{i=0}^m \left( \frac{a_i}{b_i} \right) =
  \cfrac{a_0}{
  b_0 + \cfrac{a_1}{
  b_1 + \genfrac{}{}{0pt}{0}{}{\displaystyle\ddots + \cfrac{a_m}{b_m}}}}, \qquad
 \KK_{i=0}^\infty \left( \frac{a_i}{b_i} \right) =
  \cfrac{a_0}{
  b_0 + \cfrac{a_1}{
  b_1 + \genfrac{}{}{0pt}{0}{}{\ddots }}}.
\]
\end{defn}

The following lemma will be used later. 

\begin{lem}
  \label{lem:K=K}
For any sequences $\{a_i:0\le i\le m\}$, $\{b_i:0\le i\le m\}$, and $\{c_i:-1\le i\le m\}$, we have
\[
 \KK_{i=0}^m \left( \frac{a_i}{b_i} \right) =
\frac{1}{c_{-1}} \KK_{i=0}^m \left( \frac{a_ic_{i-1}c_i}{b_ic_i} \right).
\]
\end{lem}
\begin{proof}
By multiplying $c_i$ to by the numerator and denominator of the $i^{\mathrm{th}}$ fraction,
we obtain
\[
  \cfrac{a_0}{
  b_0 + \cfrac{a_1}{
  b_1 \displaystyle+ \genfrac{}{}{0pt}{0}{}{\ddots + \cfrac{a_m}{b_m}}}}
=  \cfrac{a_0c_0}{
  b_0c_0 + \cfrac{a_1c_0c_1}{
  b_1c_1 + \genfrac{}{}{0pt}{0}{}{\ddots + \cfrac{a_mc_{m-1}c_m}{b_mc_m}}}},
\]
which is equivalent to the equation in the lemma.
\end{proof}

\begin{lem}[$q$-N\"orlund fraction]
\label{lem:norlund}
  We have
\[
\frac{\qhyper21{a,b}c{q,z}}{\qhyper21{aq,bq}{cq}{q,z}}= 
\frac{1-c-(a+b-ab-abq)z}{1-c}+\frac{1}{1-c} \KK_{m=1}^\infty
\left( \frac{c_m(z)}{e_m+d_mz} \right),
\]
where
\begin{align*}
  c_m(z)&=(1-aq^m)(1-bq^m)(cz-abq^mz^2)q^{m-1},\\
  e_m&=1-cq^m,\\
  d_m&=-(a+b-abq^m-abq^{m+1})q^m.
\end{align*}
\end{lem}

The $q$-N\"orlund fraction can be restated in the form of a continued fraction
for type $R_I$ orthogonal polynomials.

\begin{prop} [$q$-N\"orlund fraction restated]
  \label{prop:q-Norlund}
  We have
  \[
\frac{\qhyper21{aq,bq}{cq}{q,z}}{\qhyper21{a,b}c{q,z}}
  = \cfrac{1}{
    1-b_0z -\cfrac{a_1z+\lambda_1 z^2}{
      1-b_1z-\cfrac{a_2z+\lambda_2 z^2}{
        1-b_2z- \genfrac{}{}{0pt}{1}{}{\ddots}}}},
  \]
  where
  \begin{align*}
b_m & = \frac{(a+b-abq^m-abq^{m+1})q^m}{1-cq^m},\\
a_m & = -\frac{(1-aq^m)(1-bq^m)cq^{m-1}}{(1-cq^{m-1})(1-cq^m)},\\
\lambda_m & = \frac{(1-aq^m)(1-bq^m)abq^{2m-1}}{(1-cq^{m-1})(1-cq^m)}.
  \end{align*}
\end{prop}
\begin{proof}
  By taking the inverse on each side of the equation in Lemma~\ref{lem:norlund}
  we obtain
\begin{equation}
  \label{eq:bm}
\frac{\qhyper21{aq,bq}{cq}{q,z}}{\qhyper21{a,b}c{q,z}}= 
\frac{1-c}{c_0(z)} \KK_{m=0}^\infty
\left( \frac{c_m(z)}{e_m+d_mz} \right).
\end{equation}
Applying Lemma~\ref{lem:K=K} with $c_i=1/(1-cq^i)$ and $m\to\infty$ yields
\[
\frac{\qhyper21{aq,bq}{cq}{q,z}}{\qhyper21{a,b}c{q,z}}= 
\frac{(1-cq^{-1})(1-c)}{c_0(z)} \KK_{m=0}^\infty
\left( \frac{c_m(z)/(1-cq^{m-1})(1-cq^m)}
  {e_m/(1-cq^m) + d_mz/(1-cq^m)} \right),
\]
which is the same as the desired identity.
\end{proof}

\begin{proof}[Proof of Theorem~\ref{thm:thirdmom}]
  Replace \( z\) by \(z/b \), put \( a=0 \), and let \( b\to\infty \) in Proposition~\ref{prop:q-Norlund}.
  The result is Theorem~\ref{thm:thirdmom}.
\end{proof}

Note that when $b=0$ both Lemma~\ref{lem:Heine} and 
Proposition~\ref{prop:q-Norlund} give a continued fraction expression for
$$
\frac{\qhyper21{a,0}c{q,z}}{\qhyper21{aq,0}{cq}{q,z}}.
$$ 
Therefore we obtain the following theorem.

\begin{thm} 
\label{thm:genequalmom}
We have the equality of continued fractions
\[
 \cfrac{1}{1-b_0z - \cfrac{a_1z}{1-b_1z-\cfrac{a_2z}{1-b_2z-\cdots}}}
=  \cfrac{1}{1-\cfrac{\lambda_1z}{1-\cfrac{\lambda_2z}{1-\cdots}}},
\]
where
$$
  a_{n} = \frac{(aq^n-1)cq^{n-1}}{(1-cq^{n-1})(1-cq^{n})},
  \qquad
 b_{n} = \frac{aq^{n}}{1-cq^{n}},
$$
$$
  \lambda_{2n} = \frac{-cq^{2n-1}(1-aq^n)}{(1-cq^{2n-1})(1-cq^{2n})},
  \qquad
  \lambda_{2n+1} = \frac{(a-cq^n)q^{n}}{(1-cq^{2n})(1-cq^{2n+1})}.
$$
\end{thm}

When Theorem~\ref{thm:genequalmom} is interpreted as an equality 
for moment generating functions, we 
find the following generalization of Theorem~\ref{thm:equalmom} which holds
for $q$-Lommel polynomials. 

\begin{cor} Let $\lambda_n, a_n$ and $b_n$ be given by 
Theorem~\ref{thm:genequalmom}. 
The $2n^{\mathrm{th}}$ moment of the orthogonal polynomials defined by 
$p_{n+1}(x)=xp_n(x)-\lambda_n p_{n-1}(x)$ is equal to the $n^{\mathrm{th}}$ 
moment of the type $R_I$ polynomials defined by
$r_{n+1}(x)=(x-b_n)r_n(x)-a_n x r_{n-1}(x).$ 

\end{cor}

\section{Combinatorics of moments of type \( R_I \) $q$-Lommel polynomials}
\label{sec:ratios-q-lommel}

The moment generating function for type \( R_I \) polynomials is given by the
continued fraction in \eqref{eq:typeIcontfrac}. For type \( R_I \) \( q \)-Lommel polynomials
we give in this section a general combinatorial interpretation for this
infinite continued fraction in terms of parallelogram polyominoes. 
We also interpret the finite continued fraction and give an explicit rational
expression using \( q \)-Lommel polynomials.
To be specific we give a combinatorial interpretation for the ratio 
\[
  r^{(3)}_{n}(x^{-1};q^{\nu+2},q)/r^{(3)}_{n+1}(x^{-1};q^{\nu+1},q) 
\]
of (rescaled) type \( R_I \) $q$-Lommel polynomials, Theorem~\ref{thm:ratio of
  R-1}. This is a finite version of the result of Bousquet-M\'elou and Viennot
\cite{Bousquet-Melou_1992}.
The \( n\to\infty \) limit of Theorem~\ref{thm:ratio of
  R-1} yields a quotient of \( q \)-Bessel functions,
\[
  J^{(3)}_{\nu+1}(x^{1/2};q^{-1})/J^{(3)}_{\nu}(x^{1/2};q^{-1}) 
\]
which is the moment generating function for the type \( R_I \) \( q \)-Lommel
polynomials. This material appears in our unpublished manuscript
\cite[Section~5]{kim20:_ratios_hahn}.

We shall need several definitions related to parallelogram polyominoes
and Motzkin paths.

\begin{defn}
  An \emph{NE-path} is a lattice path from \( (0,0) \) to \( (a,b) \) for some
  positive integers \( a,b \) consisting of north steps \( (0,1) \) and east
  steps \( (1,0) \). A \emph{parallelogram polyomino} is a set of unit squares
  enclosed by two NE-paths with the same ending points that do not intersect
  except the starting and ending points. Denote by \( \PP \) the set of
  parallelogram polyominoes.
\end{defn}

For a parallelogram polyomino \( \alpha\in\PP \) let $U(\alpha)$ be the upper
boundary path and $D(\alpha)$ the lower boundary path, see Figure~\ref{fig:UD
  paths}. A \emph{diagonal} of \( \alpha \) is the set of squares in \(
\alpha \) whose centers are on the line \( x+y=i \) for some integer \( i \).
The \emph{size} of a diagonal is the number of squares in it. See
Figure~\ref{fig:diag}.

\begin{defn}
  We denote by \( \PP^{\le k} \) the set of parallelogram polyominoes
  in which every diagonal has size at most \( k \).
\end{defn}

\begin{figure}
  \centering
  \begin{tikzpicture}[scale=0.5]
    \draw[help lines] (0,0) grid (8,7);
    \draw[line width = 1.5pt] (0,0)--(0,4)--(3,4)--(3,6)--(4,6)--(4,7)--(8,7);
    \draw[line width = 1.5pt,dashed] (0,0)--(3,0)--(3,1)--(4,1)--(4,2)--(7,2)--(7,6)
    --(8,6)--(8,7);
    \node[fill=white] at (1.5,5) {$U(\alpha)$};
    \node[fill=white] at (5.5,1) {$D(\alpha)$};
  \end{tikzpicture}
  \caption{The boundary paths $U(\alpha)$ and $D(\alpha)$ for a parallelogram polyomino.}
  \label{fig:UD paths}
\end{figure}
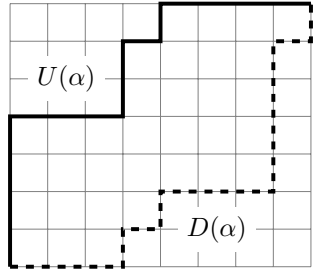

\begin{figure}
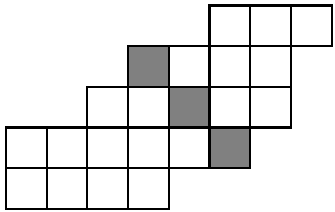

  \centering
  \begin{ytableau}
    \none&\none&    \none & \none & \none &&&\\
    \none&\none&    \none & *(gray) &&&\\
    \none&\none&    {}& &*(gray)&&\\
    {}&{}&    {}& &&*(gray)\\
    {}&{}&    {}& \\
  \end{ytableau}
  \caption{A diagonal with size \( 3 \) in a parallelogram polyomino.}
  \label{fig:diag}
\end{figure}

Consider \( \alpha\in\PP \) and a diagonal \( \tau \) of \( \alpha \). Let \( u
\) (resp.~\( d \)) be the northwest (resp.~southeast) corner of the topmost
(resp.~bottom-most) square of \( \tau \). We say that \( d \) is an
\emph{NN-diagonal} (resp.~\emph{NE-diagonal}, \emph{EN-diagonal}, and
\emph{EE-diagonal}) if the step in \( U(\alpha) \) starting at \( u \) is a
north (resp.~north, east, and east) step and the step in \( D(\alpha) \)
starting at \( d \) is a north (resp.~east, north, and east) step.
See Figure~\ref{fig:NN-diag}.

\begin{figure}
  \centering
\begin{tikzpicture}[scale=0.5]
  \draw (0,0) rectangle (1,1);
  \draw (1,-1) rectangle (2,0);
  \node at (2.3,-1.3) [circle,fill,inner sep=0.5pt]{};
  \node at (2.5,-1.5) [circle,fill,inner sep=0.5pt]{};
  \node at (2.7,-1.7) [circle,fill,inner sep=0.5pt]{};
  \draw (3,-3) rectangle (4,-2);
  \draw[line width = 2pt] (0,1)--(0,2);
  \draw[line width = 2pt] (4,-3)--(4,-2);
  \draw [decorate,decoration={brace,amplitude=5pt},xshift=-3pt,yshift=-3pt]
  (3,-3) -- (0,0) node [black,midway,xshift=-12pt, yshift=-10pt]{$n+1$};

  \begin{scope}[shift={(8,0)}]
  \draw (0,0) rectangle (1,1);
  \draw (1,-1) rectangle (2,0);
  \node at (2.3,-1.3) [circle,fill,inner sep=0.5pt]{};
  \node at (2.5,-1.5) [circle,fill,inner sep=0.5pt]{};
  \node at (2.7,-1.7) [circle,fill,inner sep=0.5pt]{};
  \draw (3,-3) rectangle (4,-2);
  \draw[line width = 2pt] (0,1)--(1,1);
  \draw[line width = 2pt] (4,-3)--(5,-3);
  \draw [decorate,decoration={brace,amplitude=5pt},xshift=-3pt,yshift=-3pt]
  (3,-3) -- (0,0) node [black,midway,xshift=-12pt, yshift=-10pt]{$n+1$};
  \end{scope}
  
  \begin{scope}[shift={(16,0)}]
  \draw (0,0) rectangle (1,1);
  \draw (1,-1) rectangle (2,0);
  \node at (2.3,-1.3) [circle,fill,inner sep=0.5pt]{};
  \node at (2.5,-1.5) [circle,fill,inner sep=0.5pt]{};
  \node at (2.7,-1.7) [circle,fill,inner sep=0.5pt]{};
  \draw (3,-3) rectangle (4,-2);
  \draw[line width = 2pt] (0,1)--(0,2);
  \draw[line width = 2pt] (4,-3)--(5,-3);
  \draw [decorate,decoration={brace,amplitude=5pt},xshift=-3pt,yshift=-3pt]
  (3,-3) -- (0,0) node [black,midway,xshift=-12pt, yshift=-10pt]{$n+1$};
  \end{scope}

  \begin{scope}[shift={(24,0)}]
  \draw (0,0) rectangle (1,1);
  \draw (1,-1) rectangle (2,0);
  \node at (2.3,-1.3) [circle,fill,inner sep=0.5pt]{};
  \node at (2.5,-1.5) [circle,fill,inner sep=0.5pt]{};
  \node at (2.7,-1.7) [circle,fill,inner sep=0.5pt]{};
  \draw (3,-3) rectangle (4,-2);
  \draw[line width = 2pt] (4,-3)--(4,-2);
  \draw[line width = 2pt] (0,1)--(1,1);
  \draw [decorate,decoration={brace,amplitude=5pt},xshift=-3pt,yshift=-3pt]
  (3,-3) -- (0,0) node [black,midway,xshift=-12pt, yshift=-10pt]{$n+1$};
  \end{scope}
\end{tikzpicture}
\caption{From left to right are shown
  an NN-diagonal, EE-diagonal, NE-diagonal, and EN-diagonal of size \( n+1 \)
  whose weights are, respectively, \( a_n,b_n,c_n \), and \( d_n \).}
  \label{fig:NN-diag}
\end{figure}
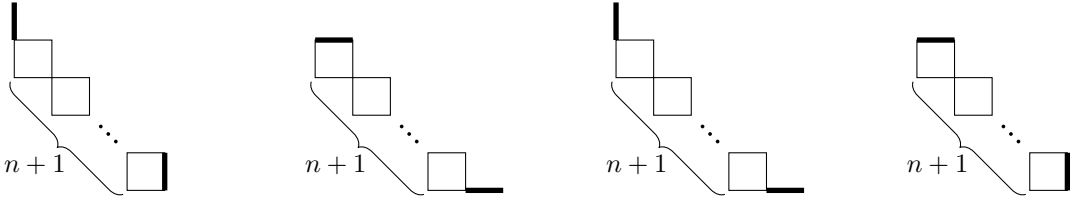

For sequences $\{a_n\}_{n\ge0}, \{b_n\}_{n\ge0}, \{c_n\}_{n\ge0}$, and $\{d_n\}_{n\ge0}$, define the
\emph{weight} $\wt(\alpha;a,b,c,d)$ of $\alpha\in\PP$ to be the product of $a_n$
(resp.~$b_n$, $c_{n}$, and $d_{n}$) for each NN-diagonal (resp.~EE-diagonal,
NE-diagonal, and EN-diagonal) of size $n+1$.

Now we review Flajolet's theory \cite{Flajolet1980} on continued fraction
expressions for Motzkin path generating functions.

\begin{defn}
A \emph{Motzkin path} is a lattice path from $(0,0)$ to $(n,0)$ consisting of up
steps $(1,1)$, down steps $(1,-1)$, and horizontal steps $(1,0)$ that never goes
below the $x$-axis. A \emph{2-Motzkin path} is a Motzkin path in which
every horizontal step is colored red or blue. The \emph{height} of a 2-Motzkin
path is the largest integer $y$ for which $(x,y)$ is a point in the path.

Denote by $\Mot_2$ the set of all 2-Motzkin paths and by
$\Mot^{\le m}_2$ the set of all 2-Motzkin paths with height at most $m$.

For sequences $\{a_n\}_{n\ge0}, \{b_n\}_{n\ge0}, \{c_n\}_{n\ge0}$, and $\{d_n\}_{n\ge0}$, define the
\emph{weight} $\wt(p;a,b,c,d)$ of a 2-Motzkin path $p$ to be the product of
$a_n$ (resp.~$b_n$, $c_{n}$, and $d_{n}$) for each red horizontal step
(resp.~blue horizontal step, up step, and down step) starting at height $n$, see
Figure~\ref{fig:Motz}.
\end{defn}

\begin{figure}
  \centering
\begin{tikzpicture}[scale=0.75]
    \draw[help lines] (0,0) grid (13,3);
    \draw[line width = 1.5pt] (0,0)-- ++(1,1)-- ++(1,1)-- ++(1,0)-- ++(1,0)-- ++(1,-1)-- ++(1,0)--
    ++(1,1)-- ++(1,1)-- ++(1,-1)-- ++(1,0)-- ++(1,-1)-- ++(1,-1)-- ++(1,0);
    \node at (.3,.7) {$c_{0}$};
    \node at (1.3,1.7) {$c_{1}$};
    \node at (6.3,1.7) {$c_{1}$};
    \node at (7.3,2.7) {$c_{2}$};
    \node at (2.5,2.4) {$a_{2}$};
    \node at (3.5,2.4) {$b_{2}$};
    \node at (4.7,1.7) {$d_{2}$};
    \node at (5.5,1.4) {$b_{1}$};
    \node at (8.7,2.7) {$d_{3}$};
    \node at (9.5,2.4) {$a_{2}$};
    \node at (10.7,1.7) {$d_{2}$};
    \node at (11.7,0.7) {$d_{1}$};
    \node at (12.5,0.4) {$b_{0}$};
    \draw[line width = 1.5pt, red] (2,2)--(3,2);
    \draw[line width = 1.5pt, blue,double] (3,2)--(4,2);
    \draw[line width = 1.5pt, blue,double] (5,1)--(6,1);
    \draw[line width = 1.5pt, red] (9,2)--(10,2);
    \draw[line width = 1.5pt, blue,double] (12,0)--(13,0);
\end{tikzpicture}
\caption{A 2-Motzkin path $p$ in $\Mot_2^{\le 3}$ with $\wt(p;a,b,c,d)=a_2^2
  b_0b_1b_2 c_0c_1^2c_2d_1d_2^2d_3$. The blue horizontal edges are represented by double
  edges.}
  \label{fig:Motz}
\end{figure}
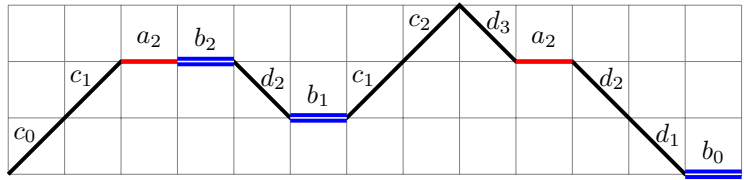

Flajolet's theory \cite{Flajolet1980} proves the following lemma
for a finite continued fraction.

\begin{lem}\label{lem:flajolet}
  Given sequences $\{a_n\}_{n\ge0}, \{b_n\}_{n\ge0}, \{c_n\}_{n\ge0}$, and $\{d_n\}_{n\ge0}$, we have
\[
\sum_{p\in\Mot_2^{\le m}} \wt(p;a,b,c,d) = \cfrac{1}{
  1-a_0-b_0 -\cfrac{c_0d_1}{
    1-a_1-b_1-
    \genfrac{}{}{0pt}{1}{}{\ddots \displaystyle - \cfrac{c_{m-1}d_m}{1-a_m-b_m}}
}}.
\]
\end{lem}

There is a well known bijection between 2-Motzkin paths and parallelogram polyominoes. 

\begin{defn} [The map $\phi:\Mot^{\le m}_2 \to \PP^{\le m+1}$]
  Let $p\in \Mot_2^{\le m}$.  
  Then $\phi(p)=\alpha$ is the parallelogram polyomino whose upper and lower
  boundary paths $U,D$ are constructed by the following algorithm.
  \begin{enumerate}
  \item The first step of $U$ (resp.~$D$) is a north (resp.~east) step.
  \item For $i=1,2,\dots,n$, where $n$ is the number of steps in $p$,
    the $(i+1)^{\mathrm{st}}$ steps of $U$ and $D$ are defined as follows.
    \begin{enumerate}
    \item If the $i^{\mathrm{th}}$ step of $p$ is an up step, then
      the $(i+1)^{\mathrm{st}}$ step of $U$ (resp.~$D$) is a north (resp.~east) step.
    \item If the $i^{\mathrm{th}}$ step of $p$ is a down step, then
      the $(i+1)^{\mathrm{st}}$ step of $U$ (resp.~$D$) is a east (resp.~north) step.
    \item If the $i^{\mathrm{th}}$ step of $p$ is a red horizontal step, then
      the $(i+1)^{\mathrm{st}}$ steps of $U$ and $D$ are both north steps.
    \item If the $i^{\mathrm{th}}$ step of $p$ is a blue horizontal step, then
      the $(i+1)^{\mathrm{st}}$ steps of $U$ and $D$ are both east steps.
    \end{enumerate}
  \item Finally, the last step of $U$ (resp.~$D$) is an east
    (resp.~north) step.
  \end{enumerate}
\end{defn}

For example, if $p$ is the $2$-Motzkin path in Figure~\ref{fig:Motz}, then
$\phi(p)$ is the parallelogram polyomino $\alpha$ in Figure~\ref{fig:UD paths}.

It is easy see from the construction that $\phi:\Mot^{\le m}_2 \to \PP^{\le
  m+1}$ is a bijection such that if \( \phi(p)=\alpha \), then \(
\wt(\alpha;a,b,c,d) = d_0 \wt(p;a,b,c,d) \).

Therefore we obtain the following proposition from Lemma~\ref{lem:flajolet},
which changes the weighted 2-Motzkin paths into weighted parallelogram
polyominoes.

\begin{prop}\label{prop:PP}
  Given sequences $\{a_n\}_{n\ge0}, \{b_n\}_{n\ge0}, \{c_n\}_{n\ge0}$, and $\{d_n\}_{n\ge0}$, we have
  \[
    \sum_{\alpha\in\PP^{\le m+1}} \wt(\alpha;a,b,c,d) = \cfrac{d_0}{
      1-a_0-b_0 -\cfrac{c_0d_1}{
        1-a_1-b_1-
        \genfrac{}{}{0pt}{1}{}{\ddots \displaystyle - \cfrac{c_{m-1}d_m}{1-a_m-b_m}}
      }}.
  \]
\end{prop}

As a special case in Proposition~\ref{prop:PP}, if $\{a_n\}_{n\ge0}, \{b_n\}_{n\ge0}, \{c_n\}_{n\ge0}$, and $\{d_n\}_{n\ge0}$ are the sequences given by $a_n=q^{n+1}Y$, $b_n=q^{n+1}X$,
$c_n=q^{n+1}XY$, and $d_n=q^{n+1}$, then one can easily check that
\[
  XY\cdot \wt(\alpha;a,b,c,d) =   X^{\col(\alpha)} Y^{\row(\alpha)} q^{\area(\alpha)}.
\]
Thus we obtain the following corollary.

\begin{cor}\label{cor:XYq}
  We have
  \[
    \sum_{\alpha\in \PP^{\le m+1}} X^{\col(\alpha)} Y^{\row(\alpha)} q^{\area(\alpha)}= \cfrac{qXY}{
      1-q(X+Y) -\cfrac{q^3XY}{
        1-q^2(X+Y)-
        \genfrac{}{}{0pt}{1}{}{\displaystyle\ddots - \cfrac{q^{2m+1}XY}{1-q^{m+1}(X+Y)}}
      }}.
  \]
\end{cor}

For the rest of this section we will find a finite version of the following
result due to Bousquet-M\'elou and Viennot \cite{Bousquet-Melou_1992}.

\begin{thm}[\cite{Delest1993} for $\nu=0$ and \cite{Bousquet-Melou_1992} for
  general $\nu$]
\label{thm:DF-BM}
The tri-variate generating function for parallelogram polyominoes is
\[
  \sum_{\alpha\in \PP} (q^\nu x)^{\col(\alpha)} (q^\nu)^{\row(\alpha)} q^{\area(\alpha)} =
 -q^\nu x^{1/2} \frac{J^{(3)}_{\nu+1}(x^{1/2};q^{-1})}{J^{(3)}_{\nu}(x^{1/2};q^{-1})}.
\]
\end{thm}

In fact Delest and F\'edou \cite{Delest1993} (for $\nu=0$), and Bousquet-M\'elou
and Viennot \cite{Bousquet-Melou_1992} state their results in the following
equivalent form:
\[
  \sum_{\alpha\in \PP} x^{\col(\alpha)} y^{\row(\alpha)} q^{\area(\alpha)} 
  =  \frac{qxy}{1-qy} \cdot \frac{\qhyper11{0}{q^2y}{q,q^2x}}{\qhyper11{0}{qy}{q,qx}} .
\]
Bousquet-M\'elou and Viennot \cite{Bousquet-Melou_1992} also showed that

\begin{equation}
\label{eq:BM1}
\sum_{\alpha\in \PP} x^{\col(\alpha)} y^{\row(\alpha)} q^{\area(\alpha)}
  = \cfrac{qxy}{
  1-q(x+y) -\cfrac{q^3xy}{
  1-q^2(x+y)- \cfrac{q^5xy}{\cdots}}}.
\end{equation}
We note that in \cite[Corollary 4.6]{Bousquet-Melou_1992} the sequence of the
coefficients of $(x+y)$ in the continued fraction \eqref{eq:BM1} was
inadvertently written $q,q^3,q^5,\dots$, where the correct sequence is
$q,q^2,q^3,\dots$. We also note that there are similar results in
\cite{Barcucci1998}.

\medskip

For a sequence \( s=\{s_n\}_{n\ge0} \), define \( \delta s=\{s_{n+1}\}_{n\ge0} \).
Kim and Stanton \cite[(5.4)]{kimstanton:R1} showed that for given sequences
$b=\{b_n\}_{n\ge0}$, $a=\{a_n\}_{n\ge0}$, and $\lambda=\{\lambda_n\}_{n\ge0}$,
and for a nonnegative integer \( k \),
\begin{equation}\label{eq:P/P=cont}
  \frac{x^m P_m(x^{-1};\delta b,\delta a,\delta \lambda)}{x^{m+1}P_{m+1}(x^{-1};b,a,\lambda)}
  = \frac{1}{-a_0x-\lambda_0x^2}\KK_{i=0}^m \left( \frac{-a_ix-\lambda_ix^2}{1-b_ix} \right).
\end{equation}

Now we are ready to prove a finite version of Theorem~\ref{thm:DF-BM}.

\begin{thm}\label{thm:ratio of R-1}
  The tri-variate generating function for bounded diagonal
  parallelogram polyominoes is
\[
  \sum_{\alpha\in \PP^{\le m+1}} (q^\nu x)^{\col(\alpha)} (q^\nu)^{\row(\alpha)}
  q^{\area(\alpha)}
  =\frac{q^{2\nu+1}}{1-q^{\nu+1}} \cdot \frac{r^{(3)}_{m}(x^{-1};q^{\nu+2},q)}{r^{(3)}_{m+1}(x^{-1};q^{\nu+1},q)}.
\]
\end{thm}
\begin{proof}
  Let $b=\{b_i\}_{i\ge0}$, $a=\{a_i\}_{i\ge0}$, and
  $\lambda=\{\lambda_i\}_{i\ge0}$, where
  \[
    b_i=\frac{q^{\nu+i+1}}{1-q^{\nu+i+1}},\qquad a_i
    =\frac{q^{2\nu+2i+1}}{(1-q^{\nu+i})(1-q^{\nu+i+1})} ,\qquad \lambda_i =0.
  \]
  Then \( P_m(x;b,a,\lambda) = r^{(3)}_{m}(x;q^{\nu+1},q) \) and 
  \(  P_m(x;\delta b,\delta a,\delta \lambda) = r^{(3)}_{m}(x;q^{\nu+2},q) \).
By \eqref{eq:P/P=cont},
\[
  \frac{r^{(3)}_{m}(x^{-1};q^{\nu+2},q)}{xr^{(3)}_{m+1}(x^{-1};q^{\nu+1},q)}
  = \frac{x^mP_{m}(x^{-1};\delta b,\delta a,\delta \lambda)}{x^{m+1}P_{m+1}(x^{-1};b,a,\lambda)}
  = \frac{1}{-a_0x}\KK_{i=0}^m \left( \frac{-a_ix}{1-b_ix} \right).
\]
By Lemma~\ref{lem:K=K} with \( c_i=1-q^{\nu+i+1} \),
\[
  \frac{1}{-a_0x}\KK_{i=0}^m \left( \frac{-a_ix}{1-b_ix} \right)
  =\frac{1}{-a_0x} \frac{1}{c_{-1}}\KK_{i=0}^m \left( \frac{-a_ic_{i-1}c_{i}x}{c_i-b_ic_ix} \right)
  =  \frac{1-q^{\nu+1}}{-q^{2\nu+1}x} \KK_{i=0}^m \left( \frac{-q^{2\nu+2i+1}x}
    {1-q^{\nu+i+1}-q^{\nu+i+1}x} \right).
\]
Letting \( X=q^\nu x \) and \( Y=q^\nu  \), and combining the above equations, we obtain
\[
\frac{q^{2\nu+1}}{1-q^{\nu+1}} \cdot \frac{r^{(3)}_{m}(x^{-1};q^{\nu+2},q)}{r^{(3)}_{m+1}(x^{-1};q^{\nu+1},q)}
=  - \KK_{i=0}^m \left( \frac{-q^{2i+1}XY} {1-q^{i+1}(X+Y)} \right).
\]
Corollary~\ref{cor:XYq} then completes the proof.
\end{proof}

By \eqref{eq:r=J}, taking the limit \( m\to\infty \) in Theorem~\ref{thm:ratio of R-1} 
we obtain Theorem~\ref{thm:DF-BM}.
We may also use Theorem~\ref{thm:RIexplicit} to write the finite 
continued fraction as an explicit rational function.

\begin{cor}\label{cor:double sum ratio of R-1}
  The tri-variate generating function for bounded diagonal
  parallelogram polyominoes is
\[
  \sum_{\alpha\in \PP^{\le n+1}} x^{\col(\alpha)} y^{\row(\alpha)}
  q^{\area(\alpha)}
  =-\frac{x\sum_{k=0}^n \sum_{a=0}^{n-k} (-1)^k x^{a} y^{-k-a}
  q^{-\binom{k}{2}-2k}\Qbinom{k+a}{a}{q^{-1}}\Qbinom{n-a}{k}{q^{-1}}}
{\sum_{k=0}^{n+1} \sum_{a=0}^{n+1-k} (-1)^k x^{a} y^{-k-a}
q^{-\binom{k}{2}-k}\Qbinom{k+a}{a}{q^{-1}}\Qbinom{n+1-a}{k}{q^{-1}}}.
\]
\end{cor}

Cigler and Krattenthaler \cite{Kratt_Hankel1} found
a different finite version of Theorem~\ref{thm:DF-BM}.

\begin{thm}\cite[Corollary 55]{Kratt_Hankel1}
  For any integer $m\ge 1$, we have
  \[
    \sum_{\alpha\in \PP_1^{\le k}} x^{\col(\alpha)} y^{\row(\alpha)}
    q^{\area(\alpha)}=
    -\frac{ y \sum_{j=1}^{k}(-1)^{j} x^{ j} q^{\binom{j+1}{2}}
      \sum_{i=0}^{k-j}(y q)^{i}\qbinom{k-i-1}{j-1} \qbinom{i+j-1}{j-1} }
    { \sum_{j=0}^{k}(-1)^{j} x^{j} q^{\binom{j+1}{2}}
      \sum_{i=0}^{k-j}(y q)^{i}\qbinom{k-i}{j} \qbinom{i+j-1}{j-1} },
  \] 
  where \( \PP_1^{\le k} \) is the set of
  parallelogram polyominoes such that each column has length at most \( k \).
\end{thm}

\begin{remark}
  The second odd-even trick \eqref{eq:mu2n+2} with \( \lambda_{2k-1}=q^ky \)
  and \( \lambda_{2k}=q^k \) gives 
  \begin{equation}
    \label{eq:BM2}
    1+\cfrac{qy}{
      1-q(x+y) -\cfrac{q^3xy}{
        1-q^2(x+y)- \cfrac{q^5xy}{\cdots}}}=
    \cfrac{1}{
      1-\cfrac{qy}{
        1-\cfrac{qx}{
          1-\cfrac{q^2y}{
            1-\cfrac{q^2x}{
              {\cdots}}}}}}.
  \end{equation}
\end{remark}

\begin{remark}

  There are also finite versions of Theorem~\ref{thm:ratio of R-1} for the
  classical \( q \)-Lommel polynomials and the even-odd $q$-Lommel polynomials.
  The rational function is again a quotient of orthogonal polynomials while the
  weights on \( \PP^{\le m+1} \) depend upon the diagonals.

  Here are the infinite continued fractions for these two cases. For the
  classical $q$-Lommel polynomials, Theorem~\ref{thm:firstmom} becomes
\begin{equation}
\label{eq:classxy}
\frac{\ _2\phi_1(0,0;q^2y;q;-qx)}{\ _2\phi_1(0,0;qy;q;-qx)}=
\cfrac{1-qy}{
  1-qy -\cfrac{q^2xy}{
  1-q^2y- \cfrac{q^3xy}{
 1-q^3y-\cfrac{q^4xy} {\cdots}}}}.
\end{equation}
For the even-odd $q$-Lommel polynomials, Theorem~\ref{thm:secondmom} becomes
\begin{equation}
\label{eq:evenoddxy}
\frac{\ _1\phi_1(0;q^2y;q;q^2x)}{\ _1\phi_1(0;qy;q;qx)}=
\cfrac{1-qy}{
  1-qy -\cfrac{A_1}{
  1-q^2y- \cfrac{A_2}{
 1-q^3y-\cfrac{A_3} {\cdots}}}}
\end{equation}
where $A_{2k-1}=xq^k$ and $A_{2k}=xyq^{3k/2+1}.$

\end{remark}

\section{Concurrence of moments}
\label{sec:concurrence-moments}

Recall the notation for the moments
$\mu_{n} \left( \{b_k\}_{k\ge0}, \{\lambda_k\}_{k\ge0} \right)$ and
$\mu_{n} \left( \{b_k\}_{k\ge0}, \{a_k\}_{k\ge0}, \{\lambda_k\}_{k\ge0} \right)$
 in Section~\ref{sec:moments-cont-fract}.
There is a concurrence of moments (see Propositions~\ref{prop:firstprop} and
\ref{prop:secondprop}),
which we call the first and second \emph{odd-even tricks}
\begin{align}
  \label{eq:mu2n}
  \mu_{2n} \left( \{0\}, \{\lambda_k\} \right)
  &=\mu_{n} \left( \{\lambda_{2k}+\lambda_{2k+1}\} , \{\lambda_{2k}\lambda_{2k-1}\} \right),\\
  \label{eq:mu2n+2}
  \mu_{2n+2} \left( \{0\}, \{\lambda_k\} \right)
  &=\lambda_1\mu_{n} \left( \{\lambda_{2k+2}+\lambda_{2k+1}\} , \{\lambda_{2k}\lambda_{2k+1}\} \right).
\end{align}

The classical orthogonal polynomial moments
are a special case of type \( R_I \) moments
$$
\mu_n(\{b_k\},\{0\}, \{\lambda_k\})=\mu_{n} \left( \{b_k\}, \{\lambda_k\} \right).
$$
There is another concurrence of moments, which follows from \cite[Corollary 3.7]{kimstanton:R1}
\begin{equation}
  \label{eq:1}
  \mu_{2n}(\{0\},\{a_k\}) = \mu_n(\{0\},\{a_k\},\{0\}).
\end{equation}

It is known \cite{kimstanton:R1} that 
a type \( R_I \) moment $\mu_n(\{b_k\},\{a_k\}, \{\lambda_k\})$ is a nonnegative polynomial in the 
recurrence coefficients. Besides \eqref{eq:1} Theorem~\ref{thm:equalmom} is
another example of classical
orthogonal polynomial moments being equal to type $R_I$ moments 
\begin{equation}\label{eq:con}
  \mu_{2n} \left( \{0\}, \{\Lambda_k\}  \right)
  =\mu_{n} \left( \{b_k\}, \{a_k\}, \{0\}  \right).
\end{equation}

The main result in this section is Theorem~\ref{thm:concurmom}, which expresses
the $\Lambda_k$ as a function of the sequences $a_k$ and $b_k$, thereby providing
the concurrence \eqref{eq:con}.

To prove Theorem~\ref{thm:concurmom} we need to recall a classical result and
notation. The Hankel determinant \cite[Theorem~4.2]{Chihara} will be used:
\[
 \det(\mu_{i+j} \left( \{b_k\}_{k\ge0}, \{\lambda_k\}_{k\ge0} \right) )_{i,j=0}^n
  = \lambda_1^n \lambda_2^{n-1} \cdots \lambda_n^1.
\]
Recall that for a sequence \( a=\{a_k\}_{k\ge0} \) we write
\( \delta a=\{a_{k+1}\}_{k\ge0} \). We also define
\( \delta^{-1} a=\{a_{k-1}\}_{k\ge0} \), where \( a_{-1}=1 \) (the value of \(
a_{-1} \) is irrelevant for our purpose).

\begin{defn}
  A \emph{Schr\"oder path} is a lattice path from \( (r,0) \) to \( (s,0) \), for
  some integers \( r,s \), consisting of northeast steps \( (1,1) \), east steps
  \( (1,0) \), and south steps \( (0,-1) \) that never goes below the \( x
  \)-axis. Given sequences \( b=\{b_k\}_{k\ge0} \) and \( a=\{a_k\}_{k\ge0} \),
  the \emph{weight} \( \wt(P) \) of a Schr\"oder path \( P \) is the product of \( b_i \) for
  each east step starting at height \( i \) and \( a_i \) for each south step
  starting at height \( i \).
\end{defn}

Our main theorem of this section is the next theorem.

\begin{thm}
\label{thm:concurmom}
Suppose that sequences \( b=\{b_k\}_{k\ge0} \), \( a=\{a_k\}_{k\ge0} \), and \( \Lambda=\{\Lambda_k\}_{k\ge0} \) satisfy
\[
  \mu_{2n} \left( \{0\}, \{\Lambda_k\}  \right)
  =\mu_{n} \left( \{b_k\}, \{a_k\}, \{0\}  \right).
\]
Then
\[
  \Lambda_1\Lambda_2\cdots \Lambda_{2n} = \frac{f_n(a,b)}{f_{n-1}(a,b)},
\]
where 
\[
f_{n}(a,b) = \sum_{p} \wt(p),
\]
and the sum is over all $n$-tuples $p=(P_0,P_1,\dots,P_n)$ of non-intersecting
Schr\"oder paths, $P_k:(-k,0)\to (k,0), 0\le k\le n.$
Moreover,
\[
  \Lambda_1\Lambda_2\cdots \Lambda_{2n-1} =  a_0^{-1}
  \frac{f_{n}(\delta^{-1}a,\delta^{-1} b)}{f_{n-1}(\delta^{-1}a,\delta^{-1} b)},
\]
and if \( a_k=b_k=1 \) then
\[
 f_{n}(\{1\},\{1\}) = 2^{\binom{n+1}{2}}.
\]
\end{thm}

\begin{proof}
Let
\begin{align*}
  \rho_n&:= \mu_{2n} \left( \{0\}, \{\Lambda_k\} \right)
          =\mu_{n} \left( \{b_k\}, \{a_k\}, \{0\} \right),\\
  \Delta_n&:=  \det(\rho_{i+j})_{0\le i,j\le n}.
\end{align*}
Using the odd-even trick $B_n=\Lambda_{2n+1}+\Lambda_{2n}$ and $\Theta_n =
\Lambda_{2n-1} \Lambda_{2n}$, we have
\[
  \rho_n = \mu_{2n} \left( \{0\}, \{\Lambda_k\} \right)= \mu_n(\{B_k\}, \{\Theta_k\}).
\]
Therefore
\[
  \Delta_n=  \det(\mu_{i+j}(\{B_k\}, \{\Theta_k\}))_{0\le i,j\le n} 
= \Theta_1^n \Theta_2^{n-1} \cdots \Theta_n^1 =  \Lambda_1^{n} \Lambda_2^{n}
  \Lambda_3^{n-1}\Lambda_4^{n-1}\cdots \Lambda_{2n-1}^{1}\Lambda_{2n}^1,
\]
which shows
$\Lambda_1\Lambda_2\cdots \Lambda_{2n} = \Delta_{n}/\Delta_{n-1}.$ 

Kim and Stanton \cite[Corollary~3.7]{kimstanton:R1} showed that \( \mu_{n} \left( \{b_k\},
  \{a_k\}, \{0\} \right) \) is the sum of weights of all Schr\"oder paths from \(
(0,0) \) to \( (n,0) \).
Since \( \Delta_n = \det(\mu_{i+j} \left( \{b_k\}, \{a_k\}, \{0\} \right))_{0\le
  i,j\le n} \), the $(n+1)\times (n+1)$ determinant $\Delta_n$ is the signed
generating function for $(n+1)$-tuples of Schr\"oder paths $(P_0,\ldots, P_n)$,
$P_k:(-k,0)\to (\sigma(k),0)$, for some permutation $\sigma$ of
$\{0,1,\ldots,n\}.$ Because there are no SE edges ($\lambda_k=0$), any two paths
which intersect do so at integer coordinates. Thus we may apply the
Lindstr\"om--Gessel--Viennot lemma of tail swapping to reduce this sum to
non-intersecting paths, $\sigma=$ identity, $P_k:(-k,0)\to (k,0).$ Thus \(
\Delta_n=f_n(a,b) \) and we obtain the identity for $\Lambda_1\Lambda_2\cdots
\Lambda_{2n}.$

Now using the second odd-even trick $B'_n=\Lambda_{2n+2}+\Lambda_{2n+1}$ and
 $\Lambda'_n = \Lambda_{2n+1} \Lambda_{2n}$, we have
$$
\rho_{n+1} = \mu_{2n+2} \left( \{0\}, \{\Lambda_k\} \right) = \Lambda_1\mu_{n}(\{B'_k\}, \{\Lambda'_k\}).
$$
Then
\begin{align*}
\Delta_n':=\det(\rho_{i+j+1})_{0\le i,j\le n-1}
  &= \Lambda_1^n \det(\mu_{i+j}(\{B'_k\}, \{\Lambda'_k\}))_{0\le i,j\le n-1}\\
&=  \Lambda_1^n \Lambda_2^{n-1}\Lambda_3^{n-1}\cdots \Lambda_{2n-2}^1\Lambda_{2n-1}^1,
\end{align*}
so 
$$
\Lambda_1\Lambda_2\cdots \Lambda_{2n-1} = \Delta_n'/\Delta_{n-1}'.
$$
As in the even case, \( \Delta'_n = \det(\mu_{i+j+1} \left( \{b_k\}, \{a_k\},
  \{0\} \right))_{0\le i,j\le n-1} \) is the generating function for
$n$-tuples non-intersecting Schr\"oder paths $p'=(P'_1,\dots, P'_n)$,
$P'_k:(-k+1,0)\to (k,0).$ For \( 1\le k\le n \), let \( P_k \) be the path from
\( (-k,-1) \) to \( (k,-1) \) obtained from \( P_k' \) by adding a northeast
step at the beginning and a south step at the end, and let \( P_0 \) be the
empty path from \( (0,-1) \) to \( (0,-1) \). This gives a bijection from
$n$-tuples non-intersecting Schr\"oder paths $p'=(P'_1,\dots, P'_n)$,
$P'_k:(-k+1,0)\to (k,0)$ to $(n+1)$-tuples non-intersecting Schr\"oder paths
$p=(P_0,P_1,\dots, P_n)$, $P_k:(-k,-1)\to (k,-1)$. Note that the starting point
of \( P_k \) has height \( -1 \), which shifts the indices of $a_k$ and $b_k$
down by one. This shows that
\[
  \Delta'_n =  a_0^{-n}\det(\mu_{i+j} \left( \{b_{k-1}\}, \{0\}, \{a_{k-1}\} \right))_{0\le
    i,j\le n}  = a_0^{-n} f_n(\delta^{-1}a, \delta^{-1}b),
\]
and we obtain the identity for \( \Lambda_1\Lambda_2\cdots \Lambda_{2n-1} \).

Finally the fact that $\Delta_n=2^{\binom{n+1}{2}}$ and 
$\Delta_n'=2^{\binom{n+1}{2}}$ if $a_k=b_k=1$ for all $k$ 
follows from \cite[Theorem 6.15, $A=B=1, C=0$]{kimstanton:R1}.
\end{proof}

The first few values of \( \Lambda_1\cdots\Lambda_k \) in Theorem~\ref{thm:concurmom} are 

\begin{align*}
  \Lambda_1
  &= a_0^{-1} \frac{f_1(\delta^{-1}a,\delta^{-1}b)}{f_{0}(\delta^{-1}a,\delta^{-1}b)}= \frac{a_1+b_0}{1},\\
  \Lambda_1\Lambda_2
  &=\frac{f_1(a,b)}{f_{0}(a,b)}= a_1 \frac{a_2+b_1}{1},\\
  \Lambda_1\Lambda_2\Lambda_3
  &= a_0^{-1} \frac{f_2(\delta^{-1}a,\delta^{-1}b)}{f_1(\delta^{-1}a,\delta^{-1}b)}\\
    &= a_1 \frac{a_1a_2a_3 + a_2^2b_0 + a_2a_3b_0 + 2a_2b_0b_1 + b_0b_1^2 + a_1a_2b_2 + a_2b_0b_2}{a_1 + b_0},\\
  \Lambda_1\Lambda_2\Lambda_3\Lambda_4
  &= \frac{f_2(a,b)}{f_1(a,b)}\\
  &= a_1a_2 \frac{a_2a_3a_4 + a_3^2b_1+ a_3a_4b_1+ 2a_3b_1b_2+  b_1b_2^2+a_2a_3b_3 + a_3b_1b_3}{a_2 + b_1}.
\end{align*}

\begin{remark}
  Eu and Fu \cite{Eu2005} used the idea relating \( \Delta_n \) and \(
  \Delta'_{n-1} \) in the proof of Theorem~\ref{thm:concurmom} to give a simple
  proof of the Aztec diamond theorem, which is equivalent to
  the result $\Delta_n=2^{\binom{n+1}{2}}$ when $a_k=b_k=1.$
\end{remark}

\section{Open problems}
\label{sec:open-problems}

Recall that Kishore's theorem is a statement about the power
series coefficients of the ratio $J_{\nu+1}(x)/J_\nu(x)$ of two Bessel
functions. 

\begin{thm}[Kishore, \cite{Kishore1964}]
  \label{thm:kishore}
  We have
  \begin{equation}
    \frac{J_{\nu+1}(z)}{J_\nu(z)}
    = \sum_{n=1}^{\infty} \frac{N_{n,\nu}}{D_{n,\nu}} \left( \frac{z}{2} \right)^{2n-1},
  \end{equation}
  where 
  \[
    D_{n,\nu} = \prod_{k=1}^n (k+\nu)^{\lfloor n/k \rfloor},
  \]
  and
  $N_{n,\nu}$ is a polynomial in $\nu$ with nonnegative integer coefficients.
\end{thm}

We conjecture the following finite version of Kishore's theorem on a
ratio of Lommel polynomials $R_{m,\nu}(x)$ defined in Section~\ref{sec:q-bessel-functions}.

\begin{conj}\label{conj:kishore}
Let 
\[
\frac{R_{m,\nu+2}(x)}{R_{m+1,\nu+1}(x)}
= \sum_{n=0}^{\infty} \frac{N^{(m)}_{n,\nu}}{D^{(m)}_{n,\nu}} \left( \frac{x}{2} \right)^{2n+1},
\]
where 
\[
D^{(m)}_{n,\nu} = \prod_{k=0}^m (\nu+k+1)^{f(m,n,k)},
\]
\[
  f(m,n,k)=
  \begin{cases}
  \max\left(\displaystyle \flr{\frac{n+1}{k+1}},\flr{\frac{n+m-2k+1}{m-k+1}} \right),
  & \mbox{if $k\ne m/2$},\\  
  1,
  & \mbox{if $k=m/2$}.
  \end{cases}
\]
Then 
$N^{(m)}_{n,\nu}$ is a polynomial in $\nu$ with nonnegative integer coefficients.
\end{conj}

\medskip

In Section~\ref{sec:moments-cont-fract} we saw that the ratio
\[
  \frac{J^{(3)}_{\nu+1}(z;q^{-1})}{J^{(3)}_{\nu}(z;q^{-1})}= \frac{-q^{\nu+1}z}{1-q^{\nu+1}}\cdot
  \frac{\qhyper11{0}{q^{\nu+2}}{q,q^{\nu+2}z^2}}
  {\qhyper11{0}{q^{\nu+1}}{q,q^{\nu+1}z^2}}
\]
has two generalizations, the $q$-N\"orlund continued fraction and Heine's
continued fraction.
These two generalizations seem to have a similar property as follows.

\begin{conj}
  Let
  \[
     \sum_{n\ge0} \gamma_n(a,b,c) z^n =  \frac{\qhyper21{aq,bq}{cq}{q,z}}{\qhyper21{a,b}{c}{q,z}}.
    \]
Then    
  \[
    \frac{\gamma_{n}(a,b,c)}{1-c} = \frac{P_n(a,b,c)}{\prod_{k=0}^n(1-cq^k)^{\flr{\frac{n+1}{k+1}}}},
  \]
  for some polynomial $P_n(a,b,c)$ in $a,b,c,q$ with integer coefficients. 
\end{conj}

\begin{conj}
  Let
  \[
     \sum_{n\ge0} \gamma'_n(a,b,c) z^n =  \frac{\qhyper21{aq,b}{cq}{q,z}}{\qhyper21{a,b}{c}{q,z}}.
    \]
Then    
  \[
    \frac{\gamma'_{n}(a,b,c)}{1-c} = \frac{P'_n(a,b,c)}{\prod_{k=0}^n(1-cq^k)^{\flr{\frac{n+1}{k+1}}}},
  \]
  for some polynomial $P'_n(a,b,c)$ in $a,b,c,q$ with integer coefficients. 
\end{conj}

\begin{problem} 
  Find a combinatorial proof of Theorem~\ref{thm:genequalmom}, which contains
  the Bousquet-M\'elou--Viennot result.
\end{problem}

\begin{problem}
  Find an Askey scheme whose top element is the associated Askey--Wilson
  polynomial which contains the \( q \)-Lommel polynomials.
\end{problem}

\end{document}